\documentclass{amsart}
\usepackage{amsmath,amsfonts,amsthm,amssymb}
\usepackage{amscd}
\usepackage{latexsym}
\usepackage{graphicx,psfrag}
\usepackage{mathrsfs}
\usepackage[all]{xy}

\newtheorem{theorem}{Theorem}[section]
\newtheorem{lemma}[theorem]{Lemma}
\newtheorem{proposition}[theorem]{Proposition}
\newtheorem{corollary}[theorem]{Corollary}

\theoremstyle{definition}
\newtheorem{definition}[theorem]{Definition}

\newtheorem{remark}[theorem]{Remark}
\newtheorem{question}[theorem]{Question}

\usepackage[colorlinks=true]{hyperref}

\newcommand{\pt}{{\mathrm{pt}}}
\newcommand{\stable}{{\mathtt{s}}}

\newcommand{\RR}{{\mathbb R}}
\newcommand{\QQ}{{\mathbb Q}}

\newcommand{\ZZ}{{\mathbb Z}}

\newcommand{\CPlane}{{\mathbb{C}\mathrm{P}^2}}

\newcommand{\SL}{\mathrm{SL}}

\newcommand{\id}{\mathrm{id}}

\newcommand{\norm}[1]{{\|{#1}\|}}
\newcommand{\seq}[1]{{\{ {#1}\}}}
\newcommand{\sat}{{\mathtt{s}}}
\newcommand{\cpn}{{\mathtt{c}}}
\newcommand{\pat}{{\mathtt{p}}}
\newcommand{\mcg}{{\mathrm{Mod}}}
\newcommand{\esg}{{\mathscr{E}}}

\newcommand{\thicktorus}{{\Theta^4}}

\title{On slope genera of knotted tori in the 4-space}

\author[Y.~Liu]{Yi Liu}
\address{%
    Department of Mathematics\\
    University of California\\
    Berkeley, CA 94720-3840, USA}
\email{%
    yliu@math.berkeley.edu}

\author[Y.~Ni]{Yi Ni}
\address{%
    Department of Mathematics\\
    California Institute of Technology\\
    Pasadena, CA 91125, USA}
\email{%
    yini@caltech.edu}

\author[H.-B.~Sun]{Hongbin Sun}
\address{%
    Department of Mathematics\\
    Princeton University\\
    Princeton, NJ 08544, USA}
\email{%
    hongbins@math.princeton.edu}

\author[S.-C.~Wang]{Shicheng Wang}
\address{%
    School of Mathematical Sciences\\
    Peking University\\
    Beijing 100871, China}
\email{%
    wangsc@math.pku.edu.cn}

\date{}

\begin{document}

\begin{abstract}
    In this note, we investigate genera of slopes of a
    knotted torus in the $4$-sphere analogous to the
    genus of a classical knot. We compare various
    formulations of this notion, and
    use this notion to study the
    extendable subgroup of the mapping class group
    of a knotted torus.
\end{abstract}

\maketitle

\tableofcontents

\section{Introduction}\label{Sec-introduction}

    In the classical knot theory,
    the genus of a knot in the $3$-sphere is a basic numerical invariant which has been well
    studied. In this note, we investigate some analogous notions for the slopes of a
    knotted torus in the $4$-sphere $S^4$. These reflect certain essential difference between knotted tori
    and knotted spheres. Similar phenomena arise in the case
    of knotted surfaces in $S^4$, but the discussion would 
    require more general treatments.
    We focus on the torus case in this note for the sake of simplicity.

    A knotted torus in $S^4$ is a locally flat subsurface homeomorphic to 
    the torus. Without loss of generality,
    we may fix a choice of marking (cf.~Subsection \ref{Subsec-marking}),
    then throughout this note,
    a \emph{knotted torus} in $S^4$ means a locally flat embedding:
        $$K:T^2\hookrightarrow S^4,$$
    from the torus to the $4$-sphere. By slightly abusing the notation,
    we often write the image of $K$ still as $K$.
    For any slope (i.e.~an essential simple closed curve)
    $c\subset K$, it makes sense to define the \emph{genus}:
        $$g_K(c),$$
    of $c$ as the smallest possible genus of all the locally flat, orientable, compact
    subsurfaces $F\hookrightarrow S^4$ whose image bounds $c$ and
    meets $K$ exactly in $c$. The genus of a slope is clearly an isotopy invariant
    of the knotted torus, and indeed,
    it is invariant under \emph{extendable automorphisms}.
    More precisely, if $\tau$ is an automorphism (i.e.~an orientation-preserving self-homeomorphism up
    to isotopy) of $T^2$ that can be extended over $S^4$ as an orientation-preserving self-homeomorphism,
    then $c$ and $\tau(c)$ must have the same genus for any slope $c\subset K$.
    It is clear that all such automorphisms form a subgroup:
        $$\esg_K\,\leq\,\mcg(T^2),$$
    of the mapping class group $\mcg(T^2)$,
    called the \emph{extendable subgroup} with respect to $K$.
    See Section~\ref{Sec-generaOfSlopes} for more details. A
    primary motivation of our study is to understand $\esg_K$ with the aid of
    the slope genera.

    Being natural as it is, the genus of a slope of a knotted torus is usually
    hard to be captured. In contrast, two weaker notions yield
    much more interesting applications. One of them is called the \emph{singular genus}
    of a slope $c$, denoted as $g^\star_K(c)$. It is defined by loosening
    the locally-flat-embedding condition on the bounding surface $F$
    above, only requiring $F\to S^4$ to be continuous. Another is called the induced \emph{seminorm} on
    $H_1(T^2)$, denoted as $\norm{\cdot}_K$. This is an analogue to the (singular)
    Thurston norm in the classical context. In Section~\ref{Sec-seminorm}, we prove an inequality
    relating the seminorms associated with the satellite construction, which is analogous
    to the classical Schubert inequality for knots in $S^3$.

    A simple observation at this point is that both the singular genus and the seminorm of
    a slope are group-theoretic notions,
    which can be rephrased by the commutator length and the stable commutator length
    in the fundamental group of the exterior of the knotted torus, respectively, (Remarks
    \ref{cl},~\ref{scl}).

    As an application of these results, we study braid satellites in Section~\ref{Sec-braidSatellites}.
    In particular, this allows us to obtain examples of knotted
    tori with finite extendable subgroups. In Section~\ref{Sec-misc},
    we exhibit examples where
    the singular genus is positive for a slope with vanishing seminorm. This implies the singular genus
    is strictly stronger than the seminorm as an invariant associated to slopes.
    We also relate the
    vanishing of the singular genus for a slope $c\subset K$ to
    the extendability of the Dehn twist $\tau_c\in\mcg(T^2)$
    along $c$ in a stable sense, (Lemma~\ref{vanishingSingGenus}).

    Section \ref{Sec-bg} surveys on results relevant to our discussion.
    A few questions related to slope genera and the extendable subgroups
    will be raised in Section~\ref{Sec-questions} for further studies.

	\bigskip\noindent\textbf{Acknowledgement}. The second author was
	partially supported by an AIM Five-Year Fellowship and NSF grant
	numbers DMS-1021956 and DMS-1103976.
	The fourth author was partially supported by grant
	No.10631060 of the National Natural Science Foundation of China. 
	The authors are grateful to Seiichi Kamada for 
	clarifying some point 
	during the devolopment of the paper and for very helpful guidance
	to the literature.
	The authors also 
	thank David Gabai, Cameron Gordon, and Charles Livingston
	for suggestions and comments,  
	and thank the referee for encouraging
	us to improve the structure of this paper.

\section{Background}\label{Sec-bg}

	This section briefly surveys on the history
	relevant to our topic in several aspects.
	We hope that it will supply the reader some context for our discussion. 
	However, the reader may safely skip this part for the moment,
	and perhaps come back later for further references. We
	thank the referee for suggesting 
	us to include some of these materials. 
		
	\subsection{Genera of knots} 
		For a classical knot $k$ in $S^3$, 
		one of the most important numerical invariant
		is its genus $g(k)$, introduced by Herbert Seifert in 1935 \cite{Seifert}.
		It is naturally defined as 
		the smallest genus among that of all possible Seifert surfaces of $k$;
		and recall that a Seifert surface of $k$ is an embedded compact
		connected surfaces in $S^3$ whose boundary is $k$. In other words,
		if $k$ is not the unknot, the smallest possible complexity of a Seifert surface
		is $2g(k)-1>0$. 
		
		In $3$-dimensional topology, a suitable generalization of this notion for
		any orientable compact $3$-manifold $M$ is the Thurston norm.
		It was introduced by William Thurston in 1986 \cite{Th}.
		Thurston discovered that the smallest possible complexity of properly embedded
		surface representatives
		for elements of $H_2(M,\partial M;\ZZ)$ can be linearly countinuously
		extended over $H_2(M,\partial M;\RR)$ to be a seminorm. 
		It is actually a norm in certain cases,
		for example, if $M$ is hyperbolic of finite volume. 
		Thurston then asked if this notion coincides with
		the one defined similarly using properly immersed surfaces, 
		which was later known as the singular Thurston norm. The question
		was answered affirmatively by David Gabai \cite{Ga} 
		using his Sutured Manifold Hierarchy. 
		As an immediate consequence, it was made clear that
		there is 
		only one notion of genus (or complexity) for classical knots, whether we 
		consider connected or
		disconnected, properly immersed or embedded Seifert surfaces.
		
		Generally speaking, the genus of a knot is quite accessible. For a $(p,q)$-torus knot, where
		$p,q$ are coprime positive integers, the genus is well known to be $(p-1)(q-1)/2$.
		For a satellite knot, the Schubert inequality yields a lower bound $(\hat{g}_\pat+|w|\cdot g_\cpn)$
		of the genus, in terms of the genus $g_\cpn$ of the companion knot, the genus
		$\hat{g}_\pat$ of the desatellite knot, and the winding number $w$ of the pattern 
		\cite{Schubert-inequality}.
		Furthermore, the genus of a knot is known to be algorithmically decidable
		\cite{Schubert-genus}. In fact, certifying an upper bound is 
		NP-complete \cite{AHT}. The genus can also be bounded and detected in terms of other
		more powerful algebraic invariants, such as the knot Floer homology \cite{OS}
		and twisted Alexander polynomials \cite{FV}.

	\subsection{Knotting and marking}\label{Subsec-marking}
		One of the classical problems in topology is the Knotting Problem, namely, ``are two embeddings
		of a given space into the $n$-space isotopic?'' Usually,
		the given space is a connected closed 
		$m$-manifold $M$ where $m<n$, and the embedding is locally flat,
		and the question can be made precise most naturally in 
		the piecewise-linear or the smooth category. When the codimension is high enough, for example,
		if $n=2m+1$ and $m>1$, all embeddings are isotopic to one another 
		so they `unknot' in this sense \cite{Wu}. However,
		below the stable range the Knotting Problem becomes very interesting, as we have already
		seen in the classical knot case.
		
		Regarding an embedding of $M^m$ into $\RR^n$ as a marking of its image, the Knotting Problem may 
		be phrased as to identify or distinguish knotting types (i.e.~isotopy classes) 
		of marked submanifolds. Somewhat more naturally,
		one can ask if two unmarked knotted submanifolds are isotopic to
		each other, or precisely, if two embeddings are isotopic up to precomposing an automophism of $M$ 
		in the given category. Suppose we have already solved the Knotting Problem, 
		then the latter question amounts to asking
		whether two markings differ only by an extendable automorphism, cf.~ \cite[Lemma 2.5]{DLWY}.
		Therefore, with or without marking does not
		make a difference if $M$ has a trivial mapping class group in the category, for example,
		in the cases of classical knots and $2$-knots, but it does in general if the extendable subgroup
		is a proper subgroup of the mapping class group, cf.~ \cite{DLWY,Hi-T2,Hi-trivial,Mo}.
		
		We refer the reader to the survey \cite{Skopenkov} for
		the Embedding Problem and the Knotting Problem in general dimensions.
		
	\subsection{Knotted surfaces}\label{Subsec-knottedSurfaceSurvey}
		The study of knotted surfaces can be suitably tagged as the mid dimensional knot theory.
		In this transitional zone between 
		the low dimensional case and the high dimensional ($2$-codimensional) case,
		we find both geometric-topological and algebraic-topological methods 
		with interesting interaction. 
		For extensive references on this topic,
		see the books \cite{Kawauchi,Hillman,CS,CKS,Kamada-braid}.
		
		With an auxiliary choice of marking, let us write a knotted surface 
		as a locally flat embedding $K:F\hookrightarrow \RR^4$,
		where $F$ is a closed surface. We can visualize a knotted surface
		by drawing a diagram obtained via a generic projection of 
		$K$ onto a $3$-subspace,
		or by displaying a motion picture of links in $\RR^3$, 
		obtained via a generic line projection
		that is Morse restricted to $K$, cf.~\cite{CS,KSS}. 
		The fundamental group of the exterior is called the knot
		group of $K$, denoted as $\pi_K$.
		Similar to the classical case, $\pi_K$ has a Wirtinger type 
		presentation in terms of its diagram \cite{Yajima-knotGroup}, and  
		$\pi_K$ can be isomorphically characterized by 
		having an Artin type presentation, described 
		in terms of 2-dimensional braids \cite{Kamada-braid}.
		
		Exteriors of knotted surfaces form an interesting family of 
		$4$-manifolds. The fundamental group of any such manifold
		is nontrivial, and it contains much information
		about the topology. For instance, it has been suspected for orientable knotted surfaces
		that having an infinite cyclic knot group
		implies unknotting, namely, 
		that $K$ bounds an embedded handlebody \cite{HK}.
		By deep methods 
		of 4-manifold topology,
		this has been confirmed for knotted spheres in the 
		topological category \cite[Theorem 11.7A]{FQ}.
		In earlier studies of knotted surface, 
		a frequent topic was to look for
		examples with prescribed properties of 
		the knot group, such as required deficiency \cite{Fox,Levine,Kanenobu}, 
		or required second homology 
		\cite{BMS,Gordon-secondHomology,Litherland-secondHomology,Maeda}. In
		some other constructions 
		of particular topological significance,
		combinatorial group theory again 
		plays an important role in the step of verification	
		\cite{Gordon-2Knot,Kamada-essentialKnot,Livingston-stablyIrreducible,Livingston-indecomposable}.
		
		Many of these constructions implement satellite knotting on various stages.
		The idea of such an operation is to replace a so-called companion knotted surface with
		another one that is embedded in the regular neighborhood the former,
		often in a more complicated pattern. 
		Basic examples of satellite knotting include
		the knot connected sum of knotted surfaces,
		and Artin's spinning construction \cite{Ar}, 
		as well as its twisted generalizations
		\cite{Ze,Litherland-deformTwist}.
		Generally speaking, satellite knotting would
		lead to an increase of genus, 
		under certain natural assumptions
		such as nonzero winding number.
		However, this can be avoided if we are just concerned
		about knotted spheres or tori 
		(cf. Subsection \ref{Subsec-satellite}).
		Like in the classical case, satellite knotting only changes
		the knot group by a van Kampen type amalgamation.
		Therefore, it is usually an approach worth considering
		if one wishes to maintain some control on the group level
		during the construction.
		As far as we are concerned,  
		the first explicit formulation of 
		the satellite construction of $n$-knots
		in literature was due to Yaichi Shinohara,
		in his 1971 paper \cite{Shinohara} 
		about generalized Alexander polynomials and signatures;
		and the satellite construction of knotted tori in $\RR^4$ first 
		appeared in Richard Litherland's 1981 paper \cite{Litherland-secondHomology},
		where he studied the second homology of the knot group.

\section{Genera of slopes}\label{Sec-generaOfSlopes}

    In this section, we introduce the genus and the singular genus for any slope
    of a knotted torus $K$ in $S^4$. 
    We provide criteria about finiteness associated to
    the extendable subgroup $\esg_K$ and the stable extendable subgroup
    $\esg^\stable_K$ of $\mcg(T^2)$ in terms of these notions. 

    \subsection{Genus and singular genus}\label{Subsec-genus}
        Let $K:T^2\hookrightarrow S^4$ be a knotted torus in $S^4$, i.e.~a locally flat embedding
        of the torus into the $4$-sphere. Let:
            $$X_K\,=\,S^4-K,$$
        be the exterior of $K$ obtained by removing an open regular neighborhood of $K$.

        \begin{lemma}\label{bounding}
            Let $F^2_g$ be the closed orientable surface of genus $g$, and
            $Y$ be a simply-connected closed $4$-manifold.
            Suppose $K:F^2_g\hookrightarrow Y$ is a null-homologous,
            locally flat embedding. Write $X=Y-K$ for the exterior of $K$ in $Y$.
            Then $\partial X$ is canonically homeomorphic to $F^2_g\times S^1$, up to isotopy,
            such that the homomorphism
            $H_1(F^2_g)\to H_1(X)$ induced by including $F^2_g$ as the first factor
            $F^2_g\times\pt$ is trivial. In particular, every essential simple closed curve $c\subset F^2_g$
            bounds a locally flat, properly embedded, orientable compact surface
            $S\hookrightarrow X_K$ with $\partial S$ embedded as $c\times\pt$.
        \end{lemma}

        \begin{proof}
            This is well-known, following from an easy homological argument.
            In fact, since $K$ is null-homologous, the normal bundle of $K$ in $Y$ is trivial,
            so $\partial X$ has a natural circle bundle structure
            $p:\partial X\to F^2_g$ over $F^2_g$
            which splits.  The splitting are given by framings of the normal bundle, which are in
            natural bijection to all the homomorphisms $\iota:H_1(F^2_g)\to H_1(\partial X)$ such that
            $p_*\circ\iota:\,H_1(F^2_g)\to H_1(F^2_g)$ is the identity. Using the Poincar\'{e} duality and
            excision, it is easy to see $H^1(X)\cong\ZZ$ and $H^1(X,\partial X)=0$. Thus
            the homomorphism $H^1(X)\to H^1(\partial X)$ is injective, and the generator of $H_1(X)$ induces
            a homomorphism $\alpha:H_1(\partial X)\to\ZZ$. It is straightforward to check that 
            $\alpha$ sends the circle-fiber of $\partial X$ to $\pm1$, 
            so the kernel of $\alpha$ projects isomorphically
            onto $H_1(F^2_g)$ via $p_*$. This gives rise to the canonical splitting $\partial X=F^2_g\times S^1$.
            It follows clearly from the construction that $H_1(F^2_g)\to H_1(X)$ is trivial.
            Moreover, if $c\times\pt$ is an essential simple closed curve on $K\times\pt$, it is homologically trivial
            in $X$, so it represents an element $[a_1,b_1]\cdots[a_k,b_k]$ in the commutator subgroup of
            $\pi_1(X)$. We take a compact orientable surface $S'$ of genus $k$ with exactly one boundary component,
            and there is a map $j: S'\to X$ sending $\partial S'$ homeomorphically onto $c\times\pt$. By a general position argument
            we may assume $j$ to be a locally flat proper immersion, and doing surgeries at double points yields a locally flat,
            properly embedded, orientable compact surface $S\hookrightarrow X$ bounded by
            $c\times \pt$.
        \end{proof}

        This allows us to make the following definition:

        \begin{definition}\label{genus}
            Let $K: T^2\hookrightarrow S^4$ be a knotted torus. For any slope, i.e.~an essential
            simple closed curve, $c\subset K$, the \emph{genus}:
                $$g_K(c),$$
            of $c$ is defined to be the minimum of the genus of $F$, as $F$ runs over all the locally flat,
            properly embedded, orientable, compact
            subsurfaces of $X_K$ bounded by $c\times\pt\subset \partial X_K$, (cf.~Lemma
            \ref{bounding}). The \emph{singular genus}:
                $$g^\star_K(c),$$
            of $c$ is defined to be the minimum of the genus of $F$, as $F$ runs over all the compact orientable
            surfaces with connected nonempty boundary such that there is a continuous map $F\to X_K$ sending
            $\partial F$ homeomorphically onto $c\times\pt$.
        \end{definition}

        \begin{remark}\label{cl}
            Recall that for a group $G$ and any element $u$ in the commutator subgroup
            $[G,G]$, the \emph{commutator length}:
                $$\mathrm{cl}(u),$$
            of $u$ is the smallest possible integer $k\geq0$ such that $u$ can be written
            as a product of commutators $[a_1,b_1]\cdots[a_k,b_k]$, where $a_i,b_i\in G$,
            and $i=1,\cdots,k$. Note that elements of $[G,G]$ that are conjugate in $G$
            have the same commutator length.
            As indicated in the proof of Lemma~\ref{bounding},
            it is clear that the singular genus $g^\star_K(c)$ is the commutator length $\mathrm{cl}(c)$,
            regarding $c$ as an element of the commutator subgroup of $\pi_1(X_K)$.
        \end{remark}

    \subsection{Extendable subgroup and stable extendable subgroup}\label{Subsec-extendableSubgroup}

        Let $\mcg(T^2)$ be the mapping class group of the torus,
        which consists of the isotopy classes
        of orientation-preserving self-homeomorphisms of $T^2$. Fixing a basis of $H_1(T^2)$,
        one can naturally identify $\mcg(T^2)$ as $\SL(2,\ZZ)$. We often refer to the elements
        of $\mcg(T^2)$ as \emph{automorphisms} of $T^2$, and do not distinguish
        elements of $\mcg(T^2)$ and their representatives.

        For any knotted torus $K:T^2\hookrightarrow
        S^4$, an automorphism $\tau\in\mcg(T^2)$ is said to be \emph{extendable} with respect
        to $K$ if $\tau$ can be extended as an orientation-preserving self-homeomorphism of $S^4$ via
        $K$. Note that this notion does not depend on the choice of the representative of $\tau$,
        cf.~\cite[Lemma 2.4]{DLWY}. It is also clear that all the extendable automorphisms
        form a subgroup of $\mcg(T^2)$.

        \begin{definition}\label{extSubgp}
            For a knotted torus $K:T^2\hookrightarrow S^4$, the \emph{extendable subgroup} with respect to
            $K$ is the subgroup of $\mcg(T^2)$ consisting of all the extendable automorphisms, denoted as:
                $$\esg_K\leq \mcg(T^2).$$
        \end{definition}

        The extendable subgroup $\esg_K$ reflects some essential difference between knotted tori and and knotted
        spheres (i.e.~$2$-knots) in $S^4$. For instance, it is known that $\esg_K$ is always a proper subgroup of
        $\mcg(T^2)$, of index at least three, (\cite{DLWY}, cf.~\cite{Mo} for the diffeomorphism extension case).
        Moreover, index three is realized by any
        unknotted embedding, namely, one 
        which bounds an embedded solid torus $S^1\times D^2$ in $S^4$, (\cite{Mo},
        cf.~\cite{Hi-trivial} for the general case of trivially embedded surfaces).
        In \cite{Hi-T2}, $\esg_K$ has been computed for the so-called
        spun $T^2$--knots and twisted spun $T^2$--knots.
        It is also clear that taking the connected sum with a knotted sphere in $S^4$
        does not change the extendable subgroup. However, for a general knotted torus in $S^4$,
        the extendable subgroup $\esg_K$ is poorly understood.
		In the following, we introduce a weaker notion called the stable extendable subgroup.
        From our point of view,
        the stable extendable subgroup is
        more closely related to the singular genera than the extendable subgroup is, cf.~Subsection
        \ref{Subsec-vanishingSingularGenus}.
        
        Suppose $K:T^2\hookrightarrow S^4$ is a knotted torus in $S^4$, and $Y$ is
        a closed simply connected $4$-manifold. There is a naturally induced
        embedding:
        	$$K[Y]:\,T^2\,\hookrightarrow\,Y,$$
        obtained by regarding $Y$ as the connected sum 
        $S^4\# Y$ and embedding $T^2$ into the first summand
        via $K$.
        This is well defined up to isotopy, and we call $K[Y]$
        the \emph{$Y$--stabilization} of $K$.
        An automorphism $\tau\in\mcg(T^2)$ is 
        said to be \emph{$Y$--stably extendable}, if $\tau$
        extends over $Y$ as an orientation-preserving self-homeomorphism
        via $K[Y]$. All such automorphisms clearly form a subgroup 
        of $\mcg(T^2)$. An automorphism $\tau\in\mcg(T^2)$ is said
        to be \emph{stably extendable}, if $\tau$ is $Y$--stably extendable for some 
        closed simply connected $4$-manifold $Y$.
        Note that if $\tau_1$ is $Y_1$--stably extendable and $\tau_2$ is $Y_2$--stably
        extendable, they are both $(Y_1\#Y_2)$--stably extendable.
        This means stably extendable automorphisms also form a subgroup of $\mcg(T^2)$.

        \begin{definition}\label{stabEsg}
            For a knotted torus $K:T^2\hookrightarrow S^4$, the 
            \emph{stable extendable subgroup} with respect to
            $K$ is the subgroup of $\mcg(T^2)$ consisting of all the 
            stably extendable automorphisms, denoted as:
                $$\esg^\stable_K\leq \mcg(T^2).$$
        \end{definition}


        \begin{proposition}\label{finiteness}
            Let $K:T^2\hookrightarrow S^4$ be a knotted torus. Then the following statements are true:
            \begin{enumerate}
                \item If the singular genus $g^\star_K(c)$ takes infinitely many distinct values as $c$ runs over all
                    the slopes of $K$, then the stable extendable subgroup $\esg^\stable_K$
                    is of infinite index in $\mcg(T^2)$;
                \item If there are at most finitely many distinct slopes $c\subset K$ with the singular genus
                    $g^\star_K(c)$ at most $C$ for every $C>0$, then the stable extendable subgroup
                    $\esg^\stable_K$ is finite.
            \end{enumerate}
        \end{proposition}

        \begin{remark} Hence the same holds for the extendable subgroup $\esg_K$.
            Using a similar argument, one can also show that the statements remain true
            when replacing $g^\star_K$ with $g_K$, and $\esg^\stable_K$ with $\esg_K$.\end{remark}

        \begin{proof}
            First observe that the singular genus of a slope is invariant
            under the action of a stably extendable automorphism, namely,
            if $\tau\in\esg^\stable_K$, then $g^\star_K(c)=g^\star_K(\tau(c))$
            for every slope $c\subset K$. This is clear because by the definition, $\tau$
            extends over $X_K'=X_K\#Y$
            as a homeomorphism $\tilde\tau:X_K'\to X_K'$,
            for some simply connected closed $4$-manifold $Y$. 
            This induces an automorphism of $\pi_1(X_K')\cong\pi_1(X_K)$,
            which preserves the commutator length of $c$, or equivalently, the singular genus
            $g^\star_K(c)$, (Remark~\ref{cl}).

            To see Statement (1), note that $\mcg(T^2)$ acts transitively on the space $\mathcal{C}$ of all the slopes
            on $T^2$. It follows immediately from the invariance of singular genera above that the cardinality of value set of
            $g^\star_K$ is at most the index $[\mcg(T^2):\esg^\stable_K]$. Thus if the range of $g^\star_K$ is infinite, the index
            of $\esg^\stable_K$ in $\mcg(T^2)$ is also infinite.

            To see Statement (2), suppose $\tau\in\esg^\stable_K$. By the assumption and the invariance of
            the singular genus under $\tau$, for any slope $c\subset K$,
            there are at most finitely many distinct slopes in the sequence $c, \tau(c),\tau^2(c),\cdots$. Thus for some integers
            $k>l\geq0$,
            $\tau^k(c)$ is isotopic to $\tau^l(c)$, or in other words, $\tau^d(c)$ is isotopic to $c$, where $d=k-l$.
            As $c$ is arbitrary, $\tau$ is a torsion element in $\mcg(T^2)$, so $\esg^\stable_K$ is a subgroup of
            $\mcg(T^2)$ consisting purely of torsion elements. It follows immediately that $\esg^\stable_K$ is a finite subgroup
            from the well-known fact that $\mcg(T^2)\cong\SL(2,\ZZ)$ is virtually torsion-free. Indeed, the index of any finite-index
            torsion-free normal subgroup of $\mcg(T^2)$ yields an upper bound of the size of $\esg^\stable_K$.
        \end{proof}

\section{Induced seminorms on $H_1(T^2;\RR)$}\label{Sec-seminorm}
    In this section, we introduce the seminorm $\norm{\cdot}_K$
    on $H_1(T^2;\RR)$ induced from any knotted torus $K:T^2\hookrightarrow
    S^4$. This may be regarded as a generalization of the (singular) Thurston norm in $3$-dimensional topology.
    We prove a Schubert-type inequality in terms of seminorms associated with satellite constructions.

    \subsection{The induced seminorm}\label{Subsec-seminorm}
        There are various ways to formulate the induced seminorm, among which we shall
        take a more topological one. Suppose $K:T^2\hookrightarrow S^4$ is a knotted torus in $S^4$.
        We shall first define the value of $\norm{\cdot}_K$ on $H_1(T^2;\ZZ)$ then extend
        linearly and continuously over $H_1(K;\RR)$.

        Recall that for a connected orientable compact surface $F$, the
        complexity of $F$ is defined as
        $\chi_{-}(F)=\max\,\seq{-\chi(F),0}$. In general, for an orientable
        compact surface $F=F_1\sqcup\cdots\sqcup F_s$, the \emph{complexity}
        of $F$ is defined as:
        $$x(F)=\sum_{i=1}^s\chi_{-}(F_i).$$

        For any $\gamma\in H_1(T^2)$, identified as an element of
        $H_1(\partial X_K)$, there exists a smooth immersion of pairs
        $(F,\partial F)\looparrowright(X_K,\partial X_K)$ such that $F$ is a
        (possibly disconnected) oriented compact surface, and that $\partial
        F$ represents $\gamma$. We define the \emph{complexity} of $\gamma$
        as:
        $$x(\gamma)=\min_{F}\,x(F),$$
        where $F$ runs through all the possible immersed surfaces as described
        above.

        The fact below follows immediately from the definition.

        \begin{lemma}\label{property-x}
            With the notation above,
            \begin{enumerate}
                \item $x(n\gamma)\leq nx(\gamma)$, for any $\gamma\in H_1(T^2)$ and any integer $n\geq 0$;
                \item $x(\gamma'+\gamma'')\leq x(\gamma')+x(\gamma'')$, for any $\gamma',\gamma''\in H_1(T^2)$.
            \end{enumerate}
        \end{lemma}

        \begin{definition}\label{norm}
            Let $K:T^2\hookrightarrow S^2$ be a knotted torus. For any $\gamma\in H_1(T^2)$, we define:
                $$\norm{\gamma}_K=\inf_{m\in\ZZ_+}\,\frac{x(m\gamma)}{m}.$$
        \end{definition}

        \begin{lemma}\label{property-norm}
            The following statements are true.
            \begin{enumerate}
                \item $\norm{n\gamma}_K=n\norm{\gamma}_K$, for any $\gamma\in H_1(T^2)$ and any integer $n\geq 0$;
                \item $\norm{\gamma'+\gamma''}_K\leq\norm{\gamma'}_K+\norm{\gamma''}_K$, for any $\gamma',\gamma''\in H_1(T^2)$.
            \end{enumerate}
        \end{lemma}

        \begin{proof}
            This follows from Lemma~\ref{property-x} and some elementary arguments. For any $\epsilon>0$, there is some $m>0$ such that
            $\norm{\gamma}_K>\frac{x(m\gamma)}{m}-\epsilon$, which by Lemma~\ref{property-x},
            $\geq\frac{x(nm\gamma)}{nm}-\epsilon\geq\frac{\norm{n\gamma}_K}{n}-\epsilon$. Let $\epsilon\to 0$, we see
            $\norm{\gamma}_K\ge\frac{\norm{n\gamma}_K}{n}$. Moreover, for any $\epsilon>0$, there exists $m>0$ such that
            $\norm{n\gamma}_K>\frac{x(mn\gamma)}m-\epsilon\geq n\norm{\gamma}_K-\epsilon$. Let $\epsilon\to 0$, we see
            $\norm{n\gamma}_K\ge n\norm{\gamma}_K$. This proves the first statement. To prove the second statement, for any $\epsilon>0$,
            there are $m',m''>0$ such that $\norm{\gamma'}_K>\frac{x(m'\gamma')}{m'}-\epsilon$,
            $\norm{\gamma''}_K>\frac{x(m''\gamma'')}{m''}-\epsilon$, so using Lemma~\ref{property-x},
            \begin{eqnarray*}
                \norm{\gamma'}_K+\norm{\gamma''}_K  &>      &\frac{x(m'\gamma')}{m'}+\frac{x(m''\gamma'')}{m''}-2\epsilon\\
                                                    &\geq   &\frac{x(m'm''\gamma')}{m'm''}+\frac{x(m'm''\gamma'')}{m'm''}-2\epsilon\\
                                                    &\geq   &\frac{x(m'm''(\gamma'+\gamma''))}{m'm''}-2\epsilon\\
                                                    &\geq   &\norm{\gamma'+\gamma''}_K-2\epsilon.
            \end{eqnarray*}
            Let $\epsilon\to 0$, we see the second statement.
        \end{proof}

        Provided Lemma~\ref{property-norm}, we can extend $\norm{\cdot}_K$
        radially over $H_1(T^2;\QQ)$, then extend continuously over
        $H_1(T^2;\RR)$. This uniquely defines a seminorm:
            $$\norm{\cdot}_K:H_1(T^2;\RR)\to [0,+\infty).$$
        Recall a seminorm on a real vector space $V$ is a function
        $\norm{\cdot}:V\to[0,+\infty)$ such that $\norm{r v}=|r|\,\norm{v}$,
        for any $r\in \RR$, $v\in V$; and that
        $\norm{v'+v''}\leq\norm{v'}+\norm{v''}$, for any $v',v''\in V$. It
        is a norm if it is in addition positive-definite, namely
        $\norm{v}=0$ if and only if $v\in V$ is zero.

        \begin{definition}\label{slopenorm}
            Let $K:T^2\hookrightarrow S^4$ be a knotted torus, and $c\subset T^2$ be a slope. Then the seminorm $\norm{c}_K$ is defined as
            $\norm{[c]}_K$, where $[c]\in H_1(T^2)$.
        \end{definition}

        \begin{remark}\label{scl}
            Recall that for a group $G$ and any element $u$ in the commutator subgroup
            $[G,G]$, the \emph{stable commutator length}:
                $$\mathrm{scl}(u)\,=\,\lim_{n\to+\infty}\,\frac{\mathrm{cl}(u^n)}{n},$$
            where $\mathrm{cl}(\cdot)$ denotes the commutator length (Remark~\ref{cl}).
            It is not hard to see that for any slope $c\subset K$, the seminorm
            $\norm{c}_K$ equals $\mathrm{scl}(c)$,
            regarding $c$ as an element of the commutator subgroup of $\pi_1(X_K)$,
            (cf.~\cite[Proposition 2.10]{Ca}).
        \end{remark}

        The observation below follows immediately from the definition and Proposition~\ref{finiteness}:

        \begin{lemma}\label{finitenessSeminorm}
            If $c\subset K$ is a slope with $\norm{c}_K>0$, then $g^\star_K(c)\geq\frac{\norm{c}_K+1}{2}$.
            Hence the stable extendable subgroup $\esg^\stable_K$ is finite if $\norm{\cdot}_K$ is nondegenerate. The same
            holds if replacing $g^\star_K$ with $g_K$ and $\esg^\stable_K$ with $\esg_K$.
        \end{lemma}

    \subsection{The satellite construction}\label{Subsec-satellite}
        The satellite construction for knotted
        tori is analogous to that of classical knots in $S^3$,
        cf.~Subsection \ref{Subsec-knottedSurfaceSurvey} for historical remarks.
                
        Fix a product structure of $T^2\cong S^1\times S^1$.
        We shall denote the standardly parametrized thickened torus as:
            $$\thicktorus=S^1\times S^1\times D^2.$$
        The standard unknotted torus $T_{\mathtt{std}}:T^2\subset S^4$ is known
        as a smoothly embedded torus such that $T_\mathtt{std}$ bounds two smoothly
        embedded solid tori $D^2\times S^1$ and $S^1\times D^2$ in $S^4$, respective to factors. It is unique up to diffeotopy of $S^4$. Let
        $K_\cpn:T^2\hookrightarrow S^4$ be a knotted torus. There is a natural
        trivial product structure on a compact tubular neighborhood $\mathcal{N}(K_\cpn)\cong T^2\times D^2$
        of $K_\cpn$, so that $c\times*$ is homologically trivial in the complement $X_{K_\cpn}$
        for any slope $c\subset T^2$. Thus there is a natural isomorphism:
            $$\mathcal{N}(K_\cpn)\cong\thicktorus,$$
        up to isotopy, as we fixed the product structure on $T^2$.

        \begin{definition}\label{def-pat}
            A \emph{pattern} knotted torus is a smooth embedding $K_\pat:T^2\hookrightarrow \thicktorus$.
            The \emph{winding number} $w(K_\pat)$ of $K_\pat$ is the algebraic intersection number
            of $[K_\pat]\in H_2(\thicktorus)$ and the fiber disk $[\pt\times \pt\times D^2]
            \in H_2(\thicktorus,\partial\thicktorus)$.
        \end{definition}

        \begin{definition}\label{def-sat}
            Let $K_\cpn:T^2\hookrightarrow S^4$ be a knotted torus and
            $K_\pat:T^2\hookrightarrow \thicktorus$ be a pattern knotted torus. After fixing a product structure on $T^2$,
            the \emph{satellite} knotted torus, denoted as:
                $$K=K_\cpn\cdot K_\pat,$$
            is the composition:
                $$T^2\stackrel{K_\pat}{\longrightarrow}\thicktorus\stackrel{\cong}{\longrightarrow}
                 \mathcal{N}(K_\cpn)\stackrel{\subset}{\longrightarrow} S^4.$$
            We call $K_\cpn$ the \emph{companion} knotted torus. The \emph{desatellite}
            $\hat{K}_\pat:T^2\hookrightarrow S^4$ of $K$ is the knotted torus $\hat{K}_\pat=T_\mathtt{std}\cdot K_\pat$.
        \end{definition}

        For any element $\gamma\in H_1(T^2)$ and a pattern $K_\pat:T^2\hookrightarrow\thicktorus$,
        there is a push-forward element $\gamma_\cpn\in H_1(T^2)$ under the composition:
            $$T^2\stackrel{K_\pat}\longrightarrow \thicktorus\stackrel\cong\longrightarrow T^2\times D^2\longrightarrow T^2,$$
        where the isomorphism respects the choice of the product structure on $T^2$, and the last map is the projection onto the $T^2$ factor.
        If $K=K_\cpn\cdot K_\pat$ is a satellite with pattern $K_\pat$, one should regard $\gamma$ as an element of $H_1(K)$, and $\gamma_\cpn$ as an
        element of $H_1(K_\cpn)$.

    \subsection{A Schubert type inequality}
        The theorem below is analogous to the Schubert inequality in the classical knot theory (\cite[Kapitel II, \S 12]{Schubert-inequality}).

        \begin{theorem}\label{Schubert}
            Suppose $K=K_\cpn\cdot K_\pat$ is a satellite knotted torus in $S^4$. Then for
            any $\gamma\in H_1(T^2;\RR)$,
                $$\norm{\gamma}_K\geq\norm{\gamma}_{\hat{K}_\pat}.$$
            Moreover, if the winding number $w(K_\pat)$ is nonzero, then:
                $$\norm{\gamma}_K\geq\norm{\gamma}_{\hat{K}_\pat}+\norm{\gamma_\cpn}_{K_\cpn}.$$
        \end{theorem}

        We prove Theorem~\ref{Schubert} in the rest of this subsection.

        Let $X_K$ be the complement of the satellite knot $K=K_\cpn\cdot K_\pat$ in $S^4$.
        The satellite construction gives a decomposition:
            $$X_K=Y\cup X_{K_\cpn},$$
        glued along the image of $\partial\thicktorus$. $Y$ is diffeomorphic
        to the complement of $K_\pat$ in $\thicktorus$, so it has two boundary components, namely the \emph{satellite boundary}
        $\partial_\sat Y$ which is $\partial X_K$, and the \emph{companion boundary}
        $\partial_\cpn Y$ which is the image of $\partial \thicktorus$.

        Similarly, the complement $X_{\hat{K}_\pat}$ can be decomposed as $Y\cup X_{T_{\tt std}}$.

        The first inequality is proved in the following lemma:

        \begin{lemma}\label{ineq-desat}
            $\norm{\gamma}_K\geq \norm{\gamma}_{\hat{K}_\pat}$.
        \end{lemma}

        \begin{proof}
            We equip $X_{K_\cpn}$ with a finite CW complex structure such that there is only one $0$-cell and the $0$-cell is contained in $\partial
            X_{K_\cpn}$, which is a subcomplex of
            $X_{K_\cpn}$. Let $X^{(q)}_{K_\cpn}$ be the union of $\partial X_{K_\cpn}$
            and the $q$-skeleton of $X_{K_\cpn}$. We may extend the identity map on $Y$ to a continuous map
            $f:Y\cup X^{(2)}_{K_\cpn}\to X_{\hat{K}_\pat}$. To see this, note that the inclusion map $\partial X_K\to X_K$
            induces a surjective map on $H_1$ for any $K: T^2\to S^4$, so the identity map on
            $\partial X_{K_\cpn}$ induces a natural isomorphism $H_1(X_{K_\cpn})\cong H_1(X_{T_{\tt std}})$.
            Every $1$-cell in $X_{K_{\cpn}}$ represents a $1$-cycle, we can extend $\id_{\partial_\cpn Y}$
            to a map $f|: X^{(1)}_{K_\cpn}\to X_{T_{\tt std}}$, so that the induced map
            $H_1(X^{(1)}_{K_\cpn})\to H_1(X_{T_{\tt std}})$ agrees with the map
            on the first homology induced by $X^{(1)}_{K_\cpn}\hookrightarrow X_{K_\cpn}$.  It is
            easy to see $X_{T_{\tt std}}\simeq S^1\vee S^2 \vee S^2$, so $\pi_1(X_{T_{\tt std}})
            \cong\ZZ$. Hence the previous $f|$ can be further extended as $f|: X^{(2)}_{K_\cpn}\to X_{T_{\tt std}}$ as the boundary
            of any $2$-cell is mapped to a null-homotopic loop in $X_{T_{\tt std}}$ by the construction.

            Thus we obtain a map $f:Y\cup X^{(2)}_{K_\cpn}\to X_{\hat{K}_\pat}$ by the map above and the identity on $Y$.
            Let $j:F\looparrowright X_K$ be an immersed compact orientable
            surface such that $j(\partial F)\subset \partial X_K$. We may assume $F$ meets $\partial_\cpn Y$ transversely.
            We homotope $j$ to $j':F \to Y\cup X^{(2)}_{K_\cpn}$. Then we obtain a map $f\circ j':F\to
            X_{\hat{K}_\pat}$ which may be homotoped to an immersion. As $F$ is arbitrary, this clearly impies
            $\norm{\gamma}_K\geq \norm{\gamma}_{\hat{K}_\pat}$ by the definition of the seminorm.
        \end{proof}

        Now we proceed to consider the case when $w(K_\pat)\neq 0$. The image of $\pt\times \pt\times\partial D^2\subset
        Y$ under the natural inclusion $Y\subset X_K$ will be denoted $\mu_\cpn$. We call $\mu_\cpn$ the \emph{companion meridian}.
        The following lemma follows immediately from the construction:

        \begin{lemma}\label{cpnMrdn}
            Identify $H_1(X_{K_\cpn})\cong\ZZ$ and $H_1(X_K)\cong\ZZ$, then
            $H_1(X_{K_\cpn})\to H_1(X_K)$ is the mulplication by $w(K_\pat)$.
        \end{lemma}

        \begin{proof}
            Note $\mu_\cpn$ represents a generator of $H_1(X_{K_\cpn})$.
            By definition of $w(K_\pat)$, $\mu_\cpn$ is homologous to $w(K_\pat)$ times the meridian of $K$.
            The lemma follows as the meridian of $K$ generates $H_1(X_K)\cong\ZZ$ by the Alexander duality.
        \end{proof}

        \begin{lemma}\label{patBdry}
            If $w(K_\pat)\neq 0$, then the inclusion map $\partial_\cpn Y\subset Y$
            induces an injective homomorphism $H_1(\partial_\cpn Y)\to H_1(Y)$. In particular, the inclusion map
            $\partial_\cpn Y\subset Y$ is $\pi_1$-injective.
        \end{lemma}

        \begin{proof}
            By the long exact sequence:
                $$\cdots\to H_2(Y,\partial_\cpn Y)\to H_1(\partial_\cpn Y)\to H_1(Y)\to \cdots,$$
            it suffices to show $H_2(Y,\partial_\cpn Y)$ is finite, since $H_1(\partial_\cpn Y)\cong H_1(\partial\thicktorus)$ is torsion-free.
            By the Poincar\'{e}--Lefschetz duality and excision,
                $$H_2(Y,\partial_\cpn Y)\cong H^2(Y,\partial_\sat Y)\cong H^2(\thicktorus,K_\pat).$$

            The long exact sequence:
                $$\cdots\to H^1(\thicktorus)\to H^1(K_\pat)\to H^2(\thicktorus,K_\pat)\to H^2(\thicktorus)\to H^2(K_p)\to\cdots,$$
            is induced by the inclusion $K_\pat\subset \thicktorus$, (or equivalently by $K_\pat:T^2\hookrightarrow\thicktorus$).
            Since $\thicktorus\simeq T^2$, $K_\pat$ induces a map $h:T^2\to T^2$. It is also clear that $w(K_\pat)$ is the degree
            of $h$. Since $w(K_\pat)\neq 0$, it is clear that the map $h^*:H^*(T^2)\to H^*(T^2)$ is injective on all dimensions, so must
            be $H^*(\thicktorus)\to H^*(K_\pat)$. Thus
            $H^2(\thicktorus,K_\pat)$ is finite from the long exact sequence. We conclude $H_2(Y,\partial_\cpn Y)$ is finite as desired.
        \end{proof}

        Note it suffices to prove Theorem~\ref{Schubert} for $\gamma\in H_1(T^2;\ZZ)$. Remember that we regard
        $\gamma$ as in $H_1(K)$, identified as the kernel of $H_1(\partial X_K)\to H_1(X_K)$.
        For any $\epsilon>0$, let $j:F\looparrowright X_K$ be a properly immersed orientable compact (possibly disconnected)
        surface, i.e. $j^{-1}(\partial X_K)=\partial F$, such that $j_*[\partial F]=m\,\gamma$ for some integer $m>0$,
        and that:
            $$\norm{\gamma}_{K}\leq \frac{x(F)}{m}<\norm{\gamma}_{K}+\epsilon.$$
        We may assume $F$ has no disk or closed component,
        so the complexity $x(F)=-\chi(F)$. We may also assume $F$ intersects $\partial_\cpn Y$ transversely,
        so $j^{-1}(\partial_\cpn Y)$ is a disjoint union of simple closed curves on $F$.
        Write $F_\pat$, $F_\cpn$ for $j^{-1}(Y)$, $j^{-1}(X_{K_\cpn})$, respectively.

        \begin{lemma}\label{pushOff}
            Suppose $w(K_\pat)\neq 0$. If $V$ is a component of $F_\pat$ such that $j(\partial V)\subset
            \partial_\cpn Y$, then there is a map $j'|:V\to\partial_\cpn Y$, such that
            $j'|_{\partial V}=j$.
        \end{lemma}

        \begin{proof}
            We may take a collection of embedded arcs
            $u_1,\cdots,u_{n}$ whose endpoints lie on $\partial V$, cutting
            $V$ into a disk $D$. This gives a cellular
            decomposition of $V$. We may first extend the map $j|_{\partial V}:\partial V\to \partial_\cpn Y$ to a map
            $j'|_{V^{(1)}}$ over the $1$-skeleton of $V$. Let $\phi:\partial D\to
            V^{(1)}$ be the attaching map. We have $j'_*\phi_*[\partial D]=
            j_*[\partial V]$ in $H_1(\partial_\cpn Y)$ by the construction. As
            $w(K_\pat)\neq 0$, by Lemma~\ref{patBdry}, $H_1(\partial_\cpn Y)\to
            H_1(Y)$ is an injective homomorphism, so $j_*[\partial V]=0$ in
            $H_1(\partial_\cpn Y)$ since it is bounded by $j_*[V]$. Thus
            $j'_*\phi_*[\partial D]=0$ in $H_1(\partial_\cpn Y)$, and hence
            $\partial D$ is null-homotopic in $\partial_\cpn Y$ under
            $j'\circ\phi$ as $\pi_1(\partial_\cpn Y)\cong H_1(\partial_\cpn Y)$,
            (remember $\partial_\cpn Y\cong \partial\thicktorus$ is a
            $3$-torus). Therefore, we may extend $j'|_{V^{(1)}}$ further over $D$
            to obtain $j':V\to\partial_\cpn Y$ as desired.
        \end{proof}

        \begin{lemma}\label{normalF}
            We may modify $j:F\looparrowright X_K$ within the interior of $F$, so that every component
            of $j^{-1}(\partial_\cpn Y)$ that is inessential on $F$ bounds a
            disk component of $j^{-1}(X_{K_\cpn})$.
        \end{lemma}

        \begin{proof}
            Let $a\subset j^{-1}(\partial_\cpn Y)$ be a component inessential on $F$, and $D\subset F$ be an embedded disk whose
            boundary is $a$. Suppose $D$ is not contained in $F_\cpn$, then $D\cap F_\pat\ne\emptyset$. Any component of $D\cap F_\pat$ must have all
            its boundary components lying on $j^{-1}(\partial_\cpn Y)$. By Lemma~\ref{pushOff}, we may redefine $j$ on these components relative to
            boundary so that they are all mapped into $X_\cpn$. After this modification and a small perturbation, either $a$ disappears from
            $j^{-1}(\partial_\cpn Y)$ (if $\partial D\subset D\cap F_\pat$), or at least one component of $j^{-1}(\partial_\cpn Y)$ in the interior of
            $D$ disappears (if $\partial D\subset D\cap F_\cpn$). Thus
            the number of inessential components of $j^{-1}(\partial_\cpn Y)$ decreases strictly
            under this modification. Therefore, after at most finitely many such modifications, every inessential component
            of $j^{-1}(\partial_\cpn Y)$ bounds a disk component of $F_\cpn$.
        \end{proof}

        Without loss of generality, we assume that $j:F\looparrowright X_K$ satisfies the conclusion of Lemma~\ref{normalF}.

        \begin{lemma}\label{covering}
            There is a finite cyclic covering $\kappa:\tilde{F}\to F$ such that for every essential component
            $a\in j^{-1}(\partial_\cpn Y)$ with $[j(a)]\neq 0$ in $H_1(X_K)$, and every component $\tilde{a}$ of $\kappa^{-1}(a)$,
            the image $j(\kappa(\tilde{a}))$ represents the same element
            in $H_1(X_K)\cong\ZZ$ up to sign.
        \end{lemma}

        \begin{proof}
            Let $a_1,\cdots,a_s$ be all the essential components $j^{-1}(\partial_\cpn Y)$ such that $[j(a_i)]\neq 0$
            in $H_1(X_K)\cong\ZZ$. Let $d>0$ be the least common multiple of all the $[j(a_i)]$'s. Consider the covering $\kappa:\tilde{F}
            \to F$ corresponding to the preimage of the subgroup $d\cdot H_1(X_K)$ under $\pi_1(F)\to\pi_1(X_K)\to H_1(X_K)$. It
            is straightforward to check that $\kappa$ satisfies the conclusion.
        \end{proof}

        Let $\kappa:\tilde{F}\to F$ be a covering as obtained in Lemma~\ref{covering}. Let $d>0$ be the degree of $\kappa$,
        so $x(\tilde{F})=d\,x(F)$.
        Clearly $j_*\kappa_*[\partial \tilde{F}]=md\,\gamma$, and also:
            $$\norm{\gamma}_{K}\leq\frac{x(\tilde{F})}{md}<\norm{\gamma}_{K}+\epsilon.$$
        Moreover, as any inessential component of $j^{-1}(\partial_\cpn Y)$ bounds a disk component of $F_\cpn$, it is clear
        that any inessential component of $(j\circ\kappa)^{-1}(\partial_\cpn Y)$ bounds a disk component of $\tilde{F}_\cpn=\kappa^{-1}(F_\cpn)$.

        Therefore, instead of using $j:F\looparrowright X_K$, we may use $j\circ\kappa:\tilde{F}\looparrowright X_K$ as well.
        From now on, we rewrite $j\circ\kappa$ as $j$, $\tilde{F}$ as $F$, and $md$ as $m$,
        so $j:F\looparrowright X_K$ satisfies the conclusions of Lemmas~\ref{normalF},~\ref{covering}.

        Let $Q\subset F_\cpn$ be the union of the disk components of $F_\cpn$.
        Let $F'_\cpn$ be $F_\cpn-Q$, and $F'_\pat$ be $F_\pat\cup Q$ (glued up along adjacent boundary
        components). We have the decompositions:
            $$F=F_\pat\cup F_\cpn=F'_\pat\cup F'_\cpn.$$
        Moreover, there is no inessential component of $\partial F'_\cpn$ by our assumption on $F$,
        so $F'_\cpn$ and $F'_\pat$ are essential subsurfaces
        of $F$ (i.e. whose boundary components are essential).

        \begin{lemma}\label{cutCxty}
            Suppose $F$ is a compact orientable surface with no disk or sphere component,
            and $E_1,E_2$ are essential compact subsurfaces of $F$ with disjoint interiors such that $F=E_1\cup E_2$.
            Then $x(F)= x(E_1)+ x(E_2)$.
        \end{lemma}

        \begin{proof}
            Note $\chi(F)=\chi(E_1)+\chi(E_2)$.
            As each $E_i$ is essential, there is no disk component of $E_i$, and by the assumption there is no sphere component,
            either. Thus, for each component $C$ of $E_i$, $x(C)=-\chi(C)$. We have $x(F)=x(E_1)+x(E_2)$.
        \end{proof}

        The desatellite term in Theorem~\ref{Schubert} comes from the following construction.

        \begin{lemma}\label{hatJ}
            Under the assumptions above, there is a properly immersed compact orientable surface
            $\hat{j}:\hat{F}'_\pat\looparrowright X_{\hat{K}_\pat}$
            such that $x(\hat{F}'_\pat)\leq x(F'_\pat)$, and that $\hat{j}_*[\partial \hat{F}'_\pat]=m\gamma$ in $H_1(T^2)$.
        \end{lemma}

        \begin{proof}
            As $F$ has been assumed to satisfy the conclusion of Lemma~\ref{covering}, there is an $\omega\in H_1(X_K)$,
            such that every component of
            $\partial_\cpn F'_\pat$ (i.e. $F'_\pat\cap j^{-1}(\partial_\cpn Y)$) represents either $\pm\omega$ or $0$, and the algebraic sum over all
            the components is zero since they bound $j(F'_\cpn)\subset X_K$. Thus we may assume there are $s$ components representing $0$, $t$
            components representing $\omega$, and $t$ components representing $-\omega$, where $s,t\geq 0$. We construct $\hat{F}'_\pat$ by attaching
            $s$ disks and $t$ annuli to $\partial_\cpn F'_\pat$, such that each disk is attached to a component representing $0$, and each annulus is
            attached to a pair of components representing opposite $\pm\omega$--classes. Let $\mathcal{D}$ be the union of attached disks, and
            $\mathcal{A}$ be the union of attached annuli. The result is a compact orientable surface
            $\hat{F}'_\pat=F'_\pat\cup\mathcal{D}\cup\mathcal{A}$, such that $\partial \hat{F}'_\pat\cong \partial F$. It is clear that
            $x(\hat{F}'_\pat)\leq x(F'_\pat\cup\mathcal{A})=x(F'_\pat)$, (cf. Lemma~\ref{cutCxty}).

            To construct $\hat{j}$, we extend the map $j|:F_\pat\to Y\subset X_{\hat{K}_\pat}=Y\cup X_{T_{\tt std}}$ over $\hat{F}_\pat=F_\pat\cup
            Q\cup \mathcal{D}\cup \mathcal{A}$, using the fact that $\pi_1(X_{T_{\tt std}})\cong H_1(X_{T_{\tt std}})\cong \ZZ$.
            Specifically, to extend the map over $Q$, let $s$ be a component of $\partial_\cpn F_\pat$ bounding a disk component of $Q$. Then
            $j_*[s]=0$ in $H_1(X_K)$. Hence it lies in the subgroup
            $H_1(T^2\times \pt)$ of $H_1(\partial\thicktorus)\cong H_1(\partial_\cpn Y)$, and by the desatellite construction,
            $\hat{j}(s)$ should also be
            null-homologous in $X_{T_{\tt std}}$. We can extend $\hat{j}$ over the disk $D\subset Q$ bounded by $s$. After extending for every component
            of $Q$, we obtain $\hat{j}|:F_\pat\cup Q\to X_{\hat{K}_\pat}$. Similarly, we may extend $\hat{j}|$ over $\mathcal{D}$.
            To extend over $\mathcal{A}$, let $A\subset\mathcal{A}$ be an attached annulus component as in the construction. Let $\partial A=s_+\sqcup
            s_-$, such that $j_*[s_{\pm}]=\pm\omega$ in $H_1(X_K)$, respectively. By the desatellite construction, $\hat{j}_*[s_{\pm}]=\pm\omega$ in
            $H_1(X_{T_{\tt std}})$. Since $\pi_1(X_{T_{\tt std}})\cong H_1(X_{T_{\tt std}})$, $\hat{j}(s_+)$ is free-homotopic to the
            orientation-reversal of $\hat{j}(s_-)$. In other words, we can extend $\hat{j}|$ over $A$. After extending for every attached annulus, we
            obtain $\hat{j}:\hat{F}'_\pat\to X_{\hat{K}_\pat}$.

            Since $\hat{j}|_{\partial\hat{F}'_\pat}$ is the same as $j|_{\partial F}$ under the natural identification, $\hat{j}_*[\partial
            \hat{F}'_\pat]=m\gamma$ in $H_1(T^2)$, (where $H_1(T^2)$ may be regarded as either $H_1(K)$ or $H_1(\hat{K}_\pat)$ under the natural
            identification). After homotoping $\hat{j}:\hat{F}'_\pat\to X_{\hat{K}_\pat}$ to a smooth immersion, we obtain the map as desired.
        \end{proof}

        The contribution of the companion term in Theorem~\ref{Schubert} basically comes from $F'_\cpn$. However, $j_*[F'_\cpn]$ does not necessarily
        represent $m\gamma_\cpn$, but may differ by a term of zero $\norm{\cdot}_{K_\cpn}$--seminorm.

        To be precise, note the image of any component of $\partial Q\subset \partial_\cpn Y$ under $j$ lies in the kernel of $H_1(\partial_\cpn Y)\to
        H_1(X_{K_\cpn})$, which we may identify with $H_1(K_\cpn)$. Thus $\alpha=j_*[\partial Q]\in H_1(\partial_\cpn Y)$ in fact
        lies in $H_1(K_\cpn)$. It is also clear that $j_*[\partial F_\cpn]=m\,\gamma_\cpn\in H_1(K_\cpn)<H_1(\partial_\cpn Y)$.
        Thus $\beta=m\,\gamma_\cpn-\alpha$ in $H_1(K_\cpn)<H_1(\partial_\cpn Y)$ is represented by $j_*[F'_\cpn]$. We have:
        	$$m\gamma_\cpn=\alpha+\beta.$$

        \begin{lemma}\label{normAlpha}
            With the notation above, $\norm{\alpha}_{K_\cpn}=0$, and hence $m\norm{\gamma_\cpn}_{K_\cpn}=\norm{\beta}_{K_\cpn}$.
        \end{lemma}

        \begin{proof}
            For any component $s\subset \partial Q$, $s$ bounds an embedded disk component $D$ of  $Q\subset F_\cpn$ by the definition of $Q$. It
            follows that $j(s)$ is null-homotopic in $X_{K_\cpn}$, and
            hence $\norm{j_*[s]}_{K_\cpn}=0$. As this works for any component of $\partial Q$, we see $\norm{\alpha}_{K_\cpn}
            =\norm{j_*[\partial Q]}_{K_\cpn}=0$. The `hence' part follows from that $\norm{\cdot}_{K_\cpn}$ is a seminorm on $H_1(K_\cpn;\RR)$.
        \end{proof}

        We are now ready to prove Theorem~\ref{Schubert}.

        \begin{proof}[{Proof of Theorem~\ref{Schubert}}]
            The first inequality follows from Lemma~\ref{ineq-desat}. In the rest, we assume $w(K_\pat)\neq 0$. Let
            $j:F\looparrowright X_K$ be a surface that $\epsilon$--approximates $\norm{\gamma}_K$ as before. We may
            assume $j$ satisfies the conclusion of Lemma~\ref{normalF} possibly after a modification. Possibly after passing
            to a finite cyclic covering of $F$, we may further
            assume $j$ satisfies the conclusion of Lemma~\ref{covering} as we have explained.
            We have the decomposition $F=F'_\pat\cup F'_\cpn$ of $F$ into essential subsurfaces, so by Lemma~\ref{cutCxty},
                $$x(F)\,=\,x(F'_\pat)+x(F'_\cpn).$$
            By Lemma~\ref{hatJ}, there is an immersed surface $\hat{j}:\hat{F}'_\pat\looparrowright X_{\hat{K}_\pat}$ representing $m\gamma$ in
            $H_1(\hat{K}_\pat)$, with $x(\hat{F}'_\pat)\leq x(F'_\pat)$, so:
                $$x(F'_\pat)\,\geq\,x(\hat{F}'_\pat)\geq m\norm{\gamma}_{\hat{K}_\pat}.$$
            By Lemma~\ref{normAlpha}, since $j|:F'_\cpn\looparrowright X_\cpn$ is an immersed surface representing $\beta$ in $H_1(K_\cpn)$,
                $$x(F'_\cpn)\,\geq\,\norm{\beta}_{K_\cpn}\,=\,m\norm{\gamma_\cpn}_{K_\cpn}.$$
            Combining the estimates above,
                $$x(F)\,\geq\, m\left(\norm{\gamma}_{\hat{K}_\pat}+\norm{\gamma_\cpn}_{K_\cpn}\right),$$
            thus:
                $$\norm{\gamma}_{\hat{K}_\pat}+\norm{\gamma_\cpn}_{K_\cpn}\,\leq\,\frac{x(F)}{m}\,<\, \norm{\gamma}_K+\epsilon.$$
            We conclude:
                $$\norm{\gamma}_{\hat{K}_\pat}+\norm{\gamma_\cpn}_{K_\cpn}\,\leq\, \norm{\gamma}_K,$$
            as $\epsilon>0$ is arbitrary.
        \end{proof}

\section{Braid satellites}\label{Sec-braidSatellites}
    In this section, we introduce and study braid satellites.

    \subsection{Braid patterns}\label{Subsec-braidPatterns}
        We shall fix a product structure on $T^2\cong S^1\times S^1$ throughout this section.
        By a \emph{braid} we shall mean an embedding $b:S^1\hookrightarrow S^1\times D^2$,
        whose image is a simple closed loop transverse to the fiber disks.
        We usually write $k_b$ for the classical knot in $S^3$ 
        associated to $b$, namely, the `satellite' knot with
        the trivial companion and the pattern $b$.
            
        There is a family of patterns arising from braids:

        \begin{definition}\label{bsat}
            Let $b:S^1\hookrightarrow S^1\times D^2$ be a braid.
            Define the \emph{standard braid pattern} $P_b$ associated to $b$ as:
                $$P_b=\id_{S^1}\times b: S^1\times S^1 \hookrightarrow \thicktorus,$$
            where $\thicktorus=S^1\times S^1\times D^2$ is the thickened torus. The \emph{standard braid torus}
            $K_b$ associated to $b$ is defined as the desatellite $T_{\tt std}\cdot P_b$.
        \end{definition}

        \begin{remark}\label{spunKnot}
            The standard braid torus $K_b$ is sometimes called
            the \emph{spun $T^2$--knot} obtained from the associated knot $k_b$.
            In \cite{Hi-T2}, the extendable subgroup $\esg_{K_b}$ has been explicitly computed.
        \end{remark}

        \begin{lemma}\label{bwn}
            If $b:S^1\hookrightarrow S^1\times D^2$ is a braid with winding number $w(b)$,
            then $w(P_b)=w(b)$. In particular, $w(P_b)\neq 0$.
        \end{lemma}

        \begin{proof}
            This follows immediately from the construction and the definition of winding numbers.
        \end{proof}

        \begin{proposition}\label{btNorm}
            Suppose $b$ is a braid whose associated knot $k_b$ is nontrivial. Then:
                $$\norm{\pt\times S^1}_{K_b}\,=\,2\,g(k_b)-1\textrm{, and } \norm{S^1\times \pt}_{K_b}\,=\,0,$$
            where $g(k_b)$ denotes the genus of $k_b$.
        \end{proposition}

        \begin{proof}
            For simplicity, we write $K_b$, $k_b$ as $K$, $k$ respectively.

            To see $\norm{\pt\times S^1}_{K}\geq 2g(k)-1$, the idea is to construct
            a map between the complements $f:X_K\to M_k$, where $X_K=S^4-K$, and $M_k=S^3-k$.
            Let $Y\subset X_K$ be the image of the complement $\thicktorus-P_b$, and $N\subset M_k$ be
            the image of the complement of $S^1\times D^2-b$. There is a natural projection
            map $f|:Y\cong S^1\times N\to N$. As $M_k-N$ is homeomorphic to the solid torus, which is an
            Eilenberg-MacLane space $K(\ZZ,1)$, it is not hard to see that $f|$ extends as a map $f:X_K\to M_k$.

            Provided this, for any properly immersed compact orientable surface $j:F\looparrowright
            X_K$ whose boundary represents $m[c]$, the norm of $[f\circ j(F)]$ is bounded below by
            the singular Thurston norm of $k$. As the singular Thurston norm equals the Thurston
            norm (cf. \cite{Ga}), which further equals $2g(k)-1$ for nontrivial knots, we
            obtain $\norm{\pt\times S^1}_K\geq 2g(k)-1$.

            To see $\norm{\pt\times S^1}_K=2g(k)-1$, it suffices to find a surface realizing the norm.
            In fact, one may first take an inclusion $\iota:\thicktorus\to S^1\times D^3$, where $\iota=\id_{S^1}\times\iota'$
            where $\iota':S^1\times D^2\to D^3$ is an standard unknotted embedding, i.e. whose core is unknotted
            in $D^3$ and $S^1\times \pt\subset S^1\times\partial D^2$ is the longitude. Then $K_b$ factorizes through
            a smooth embedding $S^1\times D^3\hookrightarrow S^4$ (unique up to isotopy) via $\iota\circ P_b$. This
            allows us to put a minimal genus Seifert surface of $k$ into $X_K$ so that it is bounded by the slope $\pt\times S^1$.
            Thus $\norm{\pt\times S^1}_K=2g(k)-1$.

            From the factorization above, we may also free-homotope $(\iota\circ P_b)(S^1\times \pt)$ to
            $S^1\times\{\pt'\}$ where $\pt'$ is a point on $\partial D^3$, via an annulus $S^1\times [\pt,\pt']$ where $[\pt,\pt']$
            is an arc whose interior lies in $D^3-k$. As $S^1\times\{\pt'\}$ bounds a disk outside the image of $S^1\times D^3$ in
            $S^4$, we see $\norm{S^1\times \pt}_K=0$.
        \end{proof}

    \subsection{Braid satellites}\label{Subsec-seminormBraidSat}
        As an application of the Schubert inequality for seminorms, we estimate $\norm{\cdot}_K$ for
        braid satellites of braid tori. We need the following notation.

        \begin{definition}\label{def-twist}
            Let $K:T^2\hookrightarrow S^4$ be a knotted torus in $S^4$, and $\tau:T^2\to T^2$
            be an automorphism of $T^2$. We define the \emph{$\tau$--twist} $K^\tau$ of $K$ to be the knotted torus:
                $$K\circ\tau:T^2\hookrightarrow S^4.$$
        \end{definition}

        It follows immediately that the seminorm changes under a twist according to the formula:
            $$\norm{\gamma}_{K^{\tau}}=\norm{\tau(\gamma)}_K.$$

        Fix a product structure $T^2\cong S^1\times S^1$ as before.
        We denote the basis vectors $[S^1\times \pt]$ and $[\pt\times S^1]$
        on $H_1(T^2;\RR)$ as $\xi$, $\eta$, respectively.
        A \emph{braid satellite} is known as some knotted torus of the form $K^\tau_b\cdot P_{b'}$, where
        $b,b'$ are braids with nontrivial associated knots, and $\tau\in\mcg(T^2)$. It is said to be a
        \emph{plumbing} braid satellite if $\tau(\xi)=\eta$ and $\tau(\eta)=-\xi$.

        \begin{proposition}\label{braid-braid}
            Suppose $b,b'$ are braids with nontrivial associated knots, and
            $\tau$ is an automorphism of $T^2$.
            Let $K$ be the satellite knotted torus $K^\tau_b\cdot P_{b'}$. Then for any $\gamma=x\,\xi+y\,\eta$
            in $H_1(T^2;\RR)$,
                $$\norm{\gamma}_K\,\geq\,(2g'-1)\cdot|y|+(2g-1)\cdot|rx+sw'y|.$$
            Here $g,g'>0$ are the genera of the associated knots of $b$, $b'$, respectively, and $w'$ is the winding number of
            $b'$, and $r$, $s$ are the intersection numbers $\xi\cdot\tau(\xi)$, $\xi\cdot\tau(\eta)$, respectively.
            Moreover, the equality is achieved if $K^\tau_b\cdot P_{b'}$ is a plumbing braid satellite.
        \end{proposition}

        We remark that one should not expect the seminorm lower bound be realized in general. For instance, in the extremal
        case when $\tau$ is the identity, $\pi_1(K)$ is exactly the knot group of the satellite of classical knots
        $k_b\cdot b'$, and the lower bound for the longitude slope is given by the classical Schubert inequality,
        which is not realized in general. However, the plumbing case is a little special. It provides examples
        of slopes on which the seminorm is not realized by the singular genus. In fact, when $c\subset K$ is
        a slope representing $x\,\xi+y\,\eta\in H_1(T^2)$, where $x,y$ are coprime odd integers, the formula
        yields that $\norm{c}_K$ is an even number, so the integer $g^\star_K(c)$ can never be $\frac{\norm{c}_K+1}2$.
        We shall give some estimate of the singular genus and the genus
        for plumbing braid satellites in Subsection~\ref{Subsec-plumbingBraidSatellite}.

        The corollary below follows immediately from Proposition~\ref{braid-braid}
        and Lemma~\ref{finitenessSeminorm}:

        \begin{corollary}\label{normExample}
            With the notation of Proposition~\ref{braid-braid},
            if $\tau$ is an automorphism of $T^2$ not fixing $\xi$ up to sign,
            then the stable extendable subgroup
            $\esg^\stable_K$ of $\mcg(T^2)$ with respect to $K$, and hence the extendable subgroup
            $\esg_K$, is finite.
        \end{corollary}

        In the rest of this subsection, we prove Proposition~\ref{braid-braid}

        \begin{lemma}\label{braid-braid-inequality}
            With the notation of Proposition~\ref{braid-braid},
                $$\norm{\gamma}_K\,\geq\,(2g'-1)\cdot|y|+(2g-1)\cdot|rx+sw'y|.$$
        \end{lemma}

        \begin{proof}
            By Lemma~\ref{bwn} and Theorem~\ref{Schubert},
                $$\norm{\gamma}_K\,\geq\,\norm{\gamma}_{K_{b'}}+\norm{\tau(\gamma_\cpn)}_{K_b}.$$
            Note that we are writing $\gamma_\cpn$ 
            with respect to $K_b\cdot P_{b'}$, 
            so the second term equals
            the corresponding term in Theorem~\ref{Schubert}
            with respect to the twisted satellite $K^\tau_b\cdot P_{b'}$
            via an obvious transformation.
            By Proposition~\ref{btNorm},
                $$\norm{\gamma}_{K_{b'}}\,=\, (2g'-1)\cdot|y|.$$
            As $b'$ is a braid, $P_{b'}:T^2\to\thicktorus\simeq T^2$ implies $\gamma_\cpn=x\,\xi+w'y\,\eta$. Write
            $\tau$ as $\left(\begin{array}{cc}p&q\\r&s\end{array}\right)$ in $\SL(2,\ZZ)$
            under the given basis $\xi,\eta$. Note it agrees with the notation $r,s$ in the statement.
            Then it is easy to compute that:
                $$\tau(\gamma_\cpn)\,=\,(px+qw'y)\,\xi+(rx+sw'y)\,\eta.$$
            By Proposition~\ref{btNorm} again,
                $$\norm{\tau(\gamma_\cpn)}_{K_b}\,=\,(2g-1)\cdot|rx+sw'y|.$$
            Combining these calculations, we obtain the estimate as desired.
        \end{proof}

        \begin{lemma}\label{plumbingSeminorm}
            With the notation of Proposition~\ref{braid-braid}, if $K$ is a plumbing braid satellite,
                $$\norm{\gamma}_K\,\leq\,(2g'-1)\cdot|y|+(2g-1)\cdot|x|.$$
        \end{lemma}

        \begin{proof}
            Because $\norm{\cdot}_K$ is a seminorm (Lemma~\ref{property-norm}), it suffices to prove $\norm{\xi}_K
            \leq 2g-1$ and $\norm{\eta}_K\leq 2g'-1$. The complement $X_K$ is the union of the companion piece
            $X_{K_b}=S^4-K_b$ and the pattern piece $Y=\thicktorus-P_{b'}$. Note that $\pi_1(X_{K_b})=\pi_1(M_{k_b})$
            where $M_{k_b}=S^3-k_b$ is the knot complement,
            and $\pi_1(Y)=\ZZ \times \pi_1(R_{b'})$ where $R_{b'}=S^1\times D^2-b'$ is the braid complement. From the construction
            it is clear that $\pi_1(Y)\to\pi_1(X_K)$ factors through the desatellite on the first factor, namely,
            $\ZZ \times \pi_1(M_{k_{b'}})$, so the commutator length of $\eta$ in $\pi_1(X_K)$ is at most that of $\eta$
            in $\pi_1(M_{k_{b'}})$, which is $2g'$. Moreover, the slope $\xi\in\partial X_K$ can be free-homotoped
            to a slope $\xi_\cpn$ on $\partial X_{K_b}$ since it is a fiber of $Y=S^1 \times R_{b'}$, and by the construction,
            it is clear that $\xi_\cpn$ represents the longitude slope of $\pi_1(\partial M_{k_b})$ in $\pi_1(M_{k_b})\cong
            \pi_1(X_{K_b})$, so the commutator length of $\xi$ in $\pi_1(X_K)$ is at most that of $\xi_\cpn$
            in $\pi_1(M_{k_b})$, which is $2g$. This proves the lemma because the commutator length equals the singular
            genus $g^\star_K$, which gives
            upper bounds for the seminorm $\norm{\cdot}_K$ on slopes, (Remark~\ref{cl} and Lemma~\ref{finitenessSeminorm}).
        \end{proof}

        Now Proposition~\ref{braid-braid} follows from Lemmas~\ref{braid-braid-inequality},~\ref{plumbingSeminorm}.
        
        \begin{remark}\label{rhombus}
        	For plumbing braid satellites, since the norm is given by $\norm{\gamma}_K=(2g'-1)|y|+(2g-1)|x|$,
        	the unit ball of the norm of plumbing satellite is the rhombus on the plane with 
        	the vertices $(\pm\frac{1}{2g-1},0)$ and $(0,\pm\frac{1}{2g'-1})$.
        \end{remark}

    \subsection{On genera of plumbing braid satellites}\label{Subsec-plumbingBraidSatellite}
        In this subsection, we estimate the singular genera and the genera of slopes for plumbing braid satellites.
        While we obtain a pretty nice estimate for
        the singular genera, with the error at most one,
        we are not sure how close our genera upper bound is to being the best possible.

        \begin{proposition}\label{plumbingGenera}
            Suppose $b,b'$ are braids with nontrivial associated knots, and
            $K$ is the plumbing braid satellite $K^\tau_b\cdot P_{b'}$. Then for
            every slope $c\subset K$, the following statements are true:
            \begin{enumerate}
                \item
                    The singular genus satisfies:
                        $$\frac{\norm{c}_K+1}2\,\leq\, g^\star_K(c)\,\leq\,\frac{\norm{c}_K+3}2.$$
                    In particular, if $c$ represents
                    $x\,\xi+y\,\eta$ with both $x$ and $y$ odd, 
                    then $g^\star_K(c)=\frac{\norm{c}_K}2+1$.
                \item
                    If $c$ represents $x\,\xi+y\,\eta$ in $H_1(T^2)$, where $x,y$ are coprime integers,
                    then the genus satisfies:
                        $$g_K(c)\,\leq\,g\cdot|x|\,+\,g'\cdot|y|\,+\,\frac{(|x|-1)(|y|-1)}2,$$
                    where $g,g'>0$ denote the genera of the associated knots $k_b,k_{b'}$ in $S^3$,
                    respectively.
            \end{enumerate}
            \end{proposition}

            We prove Proposition~\ref{plumbingGenera} in the rest of this subsection.
            We shall rewrite
            the slopes $S^1\times\pt,\,\pt\times S^1\subset T^2$ as $c_\xi,c_\eta$, respectively.

			We need the notion of Euler number to state the next lemma.
			Let $Y$ be a simply connected, closed oriented $4$-manifold, and let $K:T^2\hookrightarrow Y$ be 
	        a null-homologous knotted torus embedded in $Y$.
	        Let $X=Y-K$ be the compact exterior of the knotted torus.
	        For any locally flat, properly embedded
        	compact oriented surface with connected boundary,
        	$F\hookrightarrow X$, such that $\partial F$ is mapped homeomorphically onto a slope $c\times\pt$
        	of $K\times\pt$, (which exists by Lemma \ref{bounding},)
        	we may 
        	take a parallel copy $c\times\pt'\subset K\times\pt'$ of the slope, and perturb $F$ to be
        	another locally flat, properly embedded copy $F'\hookrightarrow X$ bounded by $c\times\pt'$, so that $F$, $F'$ are in general
        	position. The algebraic sum of the intersections between $F$ and $F'$ gives rise to an integer:
           		$$e(F;K)\in\ZZ,$$
        	which is known as the \emph{Euler number} of
        	the normal framing of $F$ induced from $K$.
        	In fact, one can check that $e(F;K)$ only depends
        	on the class $[F]\in H_2(X,K\times\pt)$.
        	If $Y$ is orientable but has no preferable choice of orientation,
        	we ambiguously speak of the Euler number up to sign.
  
            \begin{lemma}\label{niceSurfaces}
                There exist two disjoint, properly embedded, orientable compact surfaces
                $E,E'\hookrightarrow X_K$, bounded by the slopes $c_\xi\times p$, $c_\eta\times p'$
                in two parallel copies of the knotted torus
                $K\times p, K\times p'\subset \partial X$, respectively. Moreover,
                the genera of $E,E'$ are $g,g'$, respectively, and the Euler number of
                the normal framing $e(E;K)=e(E';K)=0$.
            \end{lemma}

            \begin{proof}
                Regarding $K$ as $T_{\tt std}\cdot P_b^\tau\cdot P_{b'}$, there is a natural decomposition:
                    $$X_K=X_0\cup Y\cup Y',$$
                where $X_0$ is the compact complement of the unknotted torus
                $T_{\tt std}$ in $S^4$, and $Y,Y'$ are the exteriors of $P_b,P_{b'}$ in
                the thickened torus $\thicktorus$, respectively. Moreover, $Y$ (resp.~$Y'$) has a
                natural product structure $c_\eta \times R_b$ (resp.~
                $c_\xi \times R_{b'}$), where $R_b$ (resp.~$R_{b'}$) denotes the exterior
                of the braid $b$ (resp.~$b'$) in the solid torus $S^1\times D^2$. As before, $\partial Y$
                (resp.~$\partial Y'$) has two components $\partial_\cpn Y$ and $\partial_\sat Y$
                (resp.~$\partial_\cpn Y'$ and $\partial_\sat Y'$). Thus $\partial X_0$ is glued to
                $\partial_\cpn Y$, and $\partial_\sat Y$ is glued to $\partial_\cpn Y'$,
                and $\partial_\sat Y'$ is exactly $\partial X_K$.

                The knot complement $M_{k_b}=S^3-k_b$ is the union of $R_b$ with a solid torus
                $S^1\times D^2$. From the classical knot theory,
                there is a Seifert surface $S$ of $k_b$ properly embedded in
                $M_{k_b}=S^3-k_b$ of genus $g$, and one can arrange
                $S$ so that it intersects $S^1\times D^2$ in a finite collection
                of $n\geq w$ disjoint parallel fiber disks. Thus $S_b=S\cap R_b$
                is a connected properly embedded orientable compact surface, so that $\partial S_b$
                has one component on $\partial_\sat R_b$ parallel to the longitude $s$, and
                $n$ components $c_1,\cdots,c_n$ on $\partial_\cpn R_b$ parallel to to $\pt\times\partial D^2$.
                Similarly, take a connected subsurface $S_{b'}\subset R_{b'}$ with $n'$ boundary components
                $c'_1,\cdots,c'_{n'}$ on the companion boundary, and one boundary component
                $s'$ on the satellite boundary.

                Construct a properly embedded compact annulus $E_{Y'}$ in $Y'=c_\xi \times R_{b'}$ by taking the
                product of $c_\xi$ with some arc $\alpha\subset R_{b'}-S_{b'}$, so that the two end-points lie on $\partial_\cpn R_{b'}$
                and $\partial_\sat R_{b'}$, respectively. Construct a properly embedded compact surface
                $E'_{Y'}\subset Y'$ by taking the product of $S_{b'}$ with some point in $c_\xi$.
                Similarly, construct a properly embedded compact surface $E_Y$ in $Y=c_\eta \times R_b$
                by taking a product of $S_b$ with some point in $c_\eta$; and construct
                a union of $n'$ annuli $E'_{Y}$ by taking the product of $c_\eta$ with $n'$ disjoint
                arcs $\alpha'_1,\cdots,\alpha'_{n'}$ in $R_b-S_b$, each of whose end-points lie on
                $\partial_\cpn R_b$ and $\partial_\sat R_b$, respectively. Under the gluing, we obtain
                two disjoint properly embedded surface $E_Y\cup E_{Y'}$ and $E'_Y\cup E'_{Y'}$ in $Y\cup Y'$,
                whose boundaries on $\partial_\sat Y'=\partial X_K\cong K\times S^1$ are
                $c_\xi\times\pt$ and $c_\eta\times\pt$, respectively. Moreover,
                it is clear that $\partial (E_Y\cup E_{Y'})$ has $n$ other boundary
                components on $\partial_\cpn Y=\partial X_0\cong T_{\tt std}\times S^1$, parallel
                to $c_\eta\times\pt$; and $\partial (E'_Y\cup E'_{Y'})$ has $n'$ other boundary components
                on $\partial_\cpn Y$, parallel to $c_\xi\times\pt$.

                It is not hard to see that one can cap off these other boundary components with disjoint
                properly embedded disks in $X_0$. In fact, we may regard $T_{\tt std}:T^2\hookrightarrow S^4$ as
                the composition:
                    $$T^2\,\cong\,c_\xi\times c_\eta\,\hookrightarrow\, c_\xi\times D^3\,\hookrightarrow\, S^4,$$
                where $c_\eta$ is a trivial knot in $D^3$. Thus the components of $\partial (E'_Y\cup
                E'_{Y'})$ that lie on $\partial X_0$ can be capped off in $c_\xi\times D^3$ disjointly.
                Moreover, the
                components of $\partial (E_Y\cup E_{Y'})$ lying on $\partial X_0$
                can be isotoped to the boundary of $c_\xi\times D^3$,
                so that they are all $c_\xi$--fibers.
                Because $S^4-c_\xi\times D^3$ is homeomorphic to $D^2\times S^2$, we may further cap off these
                fibers in the complement of $c_\xi\times D^3$ in $S^4$.

                It is straightforward to check that
                capping off $E_Y\cup E_{Y'}$ and $E'_Y\cup E'_{Y'}$ result in the surfaces $E$ and
                $E'$ respectively, as desired. Note that $e(E;K)$ vanishes
                because we can perturb the construction above to obtain a surface
                disjoint from $E$ bounding a slope parallel to $c_\xi\times\pt$ in
                $K\times\pt$. For the same reason, $e(E';K)=0$ as well.
            \end{proof}

            \begin{proof}[{Proof of Proposition~\ref{plumbingGenera}}]
                (1) It suffices to show the upper bound. By Lemma~\ref{niceSurfaces}, there are properly embedded
                surfaces $E,E'$ in $X_K$ bounded by $c_\xi\times\pt, c_\eta\times\pt$, respectively,
                and the complexity of $E$ and $E'$ realizes $\norm{c_\xi}_K$ and $\norm{c_\eta}_K$,
                respectively, (Proposition~\ref{braid-braid}). Suppose $c\subset K$ is a slope
                representing $x\xi+y\eta$. By the main theorem of \cite{Massey}, there exists an $|x|$--sheet connected covering space
                $\tilde{E}$ of $E$, which has exactly
                one boundary component if $x$ is odd, or two boundary components if $x$ is even.
                By the same method, there is also $\tilde{E'}$ which is connected $|y|$--sheet 
                covering $E'$ with one or two boundary components. 
                Since $x$ and $y$ are coprime, at most one of them is even,
                so $\tilde{E}\cup \tilde{E'}$ have at most three components.
                Then there are immersions of these surfaces into $X_K$, and by homotoping the image
                of their boundaries to $K\times\pt$ and taking the band sum to make them connected, we obtain an immersed
                subsurface $F\looparrowright X_K$ bounding the slope $c$. Since we need to add up to two bands to
                make the boundary of $F$ connected, this yields:
                    $$2\,g^\star_K(c)-1\,\leq\,-\chi(F)\,\leq\, (-\chi(E))\cdot|x|\,+\,(-\chi(E'))\cdot|y|+2\,=\,\norm{c}_K+2.$$
                Note that the last equality follows from Proposition \ref{braid-braid} as we assumed $K$ is
                the plumbing braid satellite. This proves the first statement. The `in particular' part
                is also clear because when $x,y$ are both odd, 
                $\norm{c}_K$ is an even number by the formula, so $\frac{\norm{c}_K}2+1$
                is the only integer satisfying our estimation.

                (2) In this case, we take $|x|$ copies of the embedded surface $E$, and $|y|$ copies of the
                embedded surface $E'$, in $X_K$. Because the Euler numbers of the normal framing are zero for $E$ and
                $E'$, we may assume these copies to be disjoint. Isotope their boundaries to $K\times\pt$ in $\partial X_K$,
                we see $|x|$ slopes parallel to $c_\xi$, and $|y|$ slopes parallel to $c_\eta$. As there are $|xy|$ intersection
                points, we take $|xy|$ band sums to obtain a properly embedded surface $F\hookrightarrow X_K$ bounding
                the slope $c$. There are $|x|+|y|-1$ bands
                that contribute to making the boundary of $F$ connected, and each of the other $|xy|-|x|-|y|+1$
                bands contributes one half to the genus of $F$. This implies:
                    $$g_K(c)\,\leq\,g(F)\,=\,g\cdot|x|\,+\,g'\cdot|y|\,+\,\frac{(|x|-1)(|y|-1)}2,$$
                as desired.
            \end{proof}

\section{Miscellaneous examples}\label{Sec-misc}
	In this section, we exhibit examples to show difference between concepts
	introduced in this note.
		
	\subsection{Slopes with vanishing seminorm but positive singular genus}\label{Subsec-incompressibleKleinBottles}
	    Note that we have already seen slopes whose singular genus do not realize
    	nonvanishing seminorm in plumbing braid satellites, (cf.~Proposition~\ref{braid-braid}).
    	There are also examples where the seminorm vanishes on some slope with positive
	    singular genus, as follows.
    	Our construction is based on the existence of incompressible knotted Klein bottles.
	
		Denote the Klein bottle as $\Phi^2$. A \emph{knotted Klein bottle} in $S^4$ is a locally flat
	    embedding $K:\Phi^2\hookrightarrow S^4$. We usually denote its image also as $K$, and the
	    exterior $X_K=S^4-K$ is obtained by removing an open regular neighborhood of $K$
	    from $S^4$ as before in the knotted torus case. We say a knotted Klein bottle $K$ is \emph{incompressible}
	    if the inclusion $\partial X_K\subset X_K$ induces an injective homomorphism between
	    the fundamental groups. There exist incompressible Klein bottles in $S^4$, see \cite[Lemma 4]{Kamada-essentialKnot}.
	
	    Incompressible knotted Klein bottles give rise to examples of
    	slopes on knotted tori which have vanishing seminorm but positive singular genus.
	
	    Specifically, let $K:\Phi^2\hookrightarrow S^4$ be an incompressible knotted Klein bottle.
	    Suppose $\kappa:T^2\to\Phi^2$ is a two-fold covering of the Klein bottle $\Phi^2$. Pertubing
	    $K\circ\kappa:\,T^2\to S^4$ in the normal direction of $K$ gives rise to a knotted torus:
        	$$\tilde{K}:\,T^2\hookrightarrow S^4.$$
	
	    \begin{lemma}\label{seminormVersusGenus}
    	    With the notation above, $\tilde{K}$ has a slope $c$ such that $\norm{c}_{\tilde{K}}=0$,
	        but $g^\star_{\tilde{K}}(c)>0$.
	    \end{lemma}
	
	    \begin{proof}
        	Let $\alpha\subset\Phi^2$ be an essential simple closed curve on $K$ so that $\kappa^{-1}(\alpha)$
        	has two components $c,c'\subset T^2$. Then $c,c'$ are parallel on $T^2$.
        	We choose orientations on $c,c'$ so that they are parallel as oriented curves.
        	Let $\mathcal{N}(K)$ be a compact
        	regular neighborhood of $K$ so that $Y=\mathcal{N}(K)-\tilde{K}$ is a pair-of-pants bundle
        	over $K$. Then $c$ is freely homotopic to the orientation-reversal of $c'$ within
        	$Y$. This implies that $2\,[c\times\pt]\in H_1(X_{\tilde K})$ is
        	represented by a properly immersed annulus $A\looparrowright X_{\tilde{K}}$ whose boundary
        	with the induced orientation equals $c\cup c'$. Therefore, $\norm{c}_K$ equals zero.
        	However, note that $X_{\tilde{K}}=X_{K}\cup Y$, glued
        	along $\partial X_K=\partial\mathcal{N}(K)$. Since $K$ is incompressible, $\partial X_K$
        	is $\pi_1$--injective in $X_K$. It is also clear that both components
        	of $\partial Y$ are $\pi_1$--injective in $Y$. It follows that $\pi_1(Y)$ injects into
        	$\pi_1(X_{\tilde{K}})$, and also that $\pi_1(\partial X_{\tilde{K}})$ injects into $\pi_1(X_{\tilde{K}})$. Therefore,
        	the slope $c\times\pt$ in $\partial X_{\tilde{K}}\cong \tilde{K}\times S^1$ is homotopically nontrivial in
        	$\pi_1(X_{\tilde{K}})$, so $g^\star_{\tilde K}(c)$ cannot be zero.
    	\end{proof}
	
	\subsection{Stably extendable but not extendable automorphisms}\label{Subsec-vanishingSingularGenus}
	    It is clear that the stable extendable subgroup $\esg^\stable_K$ contains the
	    extendable subgroup $\esg_K$ for any knotted torus $K:T^2\hookrightarrow S^4$.
	    They are in general not equal.
	    In fact, we show that Dehn twist along a slope with 
 		vanishing singular genus is stably extendable (Lemma \ref{vanishingSingGenus}). 
  		In particular, it follows that
  		for any unknotted embedded torus $K$, 
  		the stable extendable subgroup $\esg^\stable_K$ equals
  		$\mcg(T^2)$.
  		However, in this case,
  		the extendable subgroup $\esg_K$ is a proper subgroup of
  		$\mcg(T^2)$ of index three \cite{DLWY,Mo}.
  		Thus there are many automorphisms that are stably extendable but
  		not extendable for the unknotted embedding.

        Fix an orientation of the torus $T^2$. For any slope $c\subset T^2$ on the torus,
        we denote the (right-hand) Dehn twist along $c$ as:
            $$\tau_c:\,T^2\to T^2.$$
        More precisely, the induced automorphism on $H_1(T^2)$ is given by
        $\tau_{c*}(\alpha)\,=\,\alpha\,+\,I([c],\alpha)\,[c]$,
        for all $\alpha\in H_1(T^2)$, where $I:H_1(T^2)\times H_1(T^2)\to\ZZ$ denotes the intersection form.
        Note that
        the expression is independent from the choice of the direction of $c$.

        The criterion below is inspired from techniques in
        the paper of Susumu Hirose and Akira Yasuhara \cite{HY}. However,
        the reader should beware that our notion of stabilization in this paper
        does not change the fundamental group of the complement,
        so it is slightly different from their definition.

        \begin{lemma}\label{vanishingSingGenus} 
        	Let $K:T^2\hookrightarrow S^4$ be a knotted torus.
        	Suppose $c\subset T^2$ is a slope with the singular genus $g^\star_K(c)=0$. Then the Dehn twist
        	$\tau_c\in\mcg(T^2)$ along $c$ belongs to the stable extendable subgroup $\esg^\stable_K$.
        \end{lemma}

        \begin{proof}
	        The idea of this criterion is that, for
	        a closed simply connected oriented $4$-manifold $Y$,
	        to have the Dehn twist $\tau_c$
	        extendable over $Y$ via the $Y$--stabilization
	        $K[Y]:T^2\hookrightarrow Y$, we need $c$ to bound a locally flat, properly embedded disk
	        of Euler number $\pm1$ in the complement of $K[Y]$ in $Y$.
	        Such a $Y$ can always be chosen to be the connected sum
	        of copies of $\CPlane$ or $\overline\CPlane$.
	        	        	        	        
	        Recall that we introduced the Euler number of a surface bounding
	        a slope in Subsection \ref{Subsec-plumbingBraidSatellite}
	        before the statement of Lemma \ref{niceSurfaces}.
	        Suppose $D$ is a locally flat, properly embedded disk in $X=Y-K[Y]$ 
	        bounded by a slope $c\times\pt$
            on $K[Y]\times\pt\subset\partial X$ with $e(D;K[Y])=\pm 1$.
            We claim in this case the Dehn twist $\tau_c\in\mcg(T^2)$ along $c$
            can be extended as an orientation-preserving self-homeomorphism of $Y$.
        	In fact,
        	following the arguments in the proof of \cite[Theorem 4.1]{HY}, 
        	we may take the compact normal disk bundle $\nu_D$ of $D$,
            identified as embedded in $X$ such that $\nu_D\cap (K[Y]\times \pt)$ is an
            interval subbundle of $\nu_D$ over $\partial D$. 
            Then $e(D;K[Y])=\pm 1$ implies that $\nu_D\cap (K[Y]\times \pt)$ 
            is a (positive or negative) Hopf band in the $3$-sphere $\partial \nu_D$, whose core is $c\times\pt$.
            Thus $\tau_c$ extends over $Y$ as a self-homeomorphism
            by \cite[Proposition 2.1]{HY}.
                
            Now it suffices to find a $Y$ fulfilling
            the assumption of the
            claim above. 
            Suppose $c\subset K$ is a slope with the singular genus $g^\star_K(c)=0$, then there is a map $j:D^2\to X_K$
            so that $\partial D^2$ is mapped homeomorphically onto $c\times\pt$ in $\partial X_K\cong K\times S^1$. We may also
            assume $j$ to be an immersion by the general position argument. Blowing up all the double points of $j(D^2)$, we obtain an embedding
            $j':D^2\hookrightarrow X_K\#(\overline{\CPlane})^{\#r}$ for some integer $r\geq0$.
            Suppose $e(j'(D);K[(\overline{\CPlane})^{\#r}])$ equals $s\in\ZZ$. 
            If $s>1$, we may further blow up $s-1$ points in
            $j'(D)\subset X_K\#(\overline{\CPlane})^{\#r}$. This gives rise to $j'':D^2\hookrightarrow X_K\#
            (\overline\CPlane)^{\#(r+s-1)}$ satisfying the assumption of the claim, so the Dehn twist
            $\tau_c$ is extendable over $X=X_K\#(\CPlane)^{\#(r+s-1)}$, or in other words, 
            it is $Y$--stably extendable where $Y=(\overline\CPlane)^{\#(r+s-1)}$.
            If $s<1$, a similar argument using negative blow-ups shows that $\tau_c$
            is $Y$--stably extendable, where $Y=(\CPlane)^{\#(1-s)}\#(\overline\CPlane)^{\#r}$.
        \end{proof}

\section{Further questions}\label{Sec-questions}

    In conclusion, for a knotted torus $K:T^2\hookrightarrow S^4$, the seminorm and the singular genus of a slope
    are meaningful numerical invariants which are sometimes possible to control using group theoretic methods.
    However, the genera of slopes seem to be much harder to compute. It certainly deserves further exploration
    how to combine the group-theoretic methods with the classical $4$-manifold techniques when the fundamental group
    comes into play.

    We propose several further questions about genera, seminorm and extendable subgroups.
    Suppose $K:T^2\hookrightarrow S^4$ is a knotted torus.

    \begin{question}
        When is the unit disk of the seminorm $\norm{\cdot}_K$
        a finite rational polygon, i.e.~bounded by finitely many
        segments of rational lines? (Cf.~Remark \ref{rhombus}.)
    \end{question}

    \begin{question}
        If the index of the extendable subgroup $\esg_K$ in $\mcg(T^2)$ equals three,
        is $K$ necessarily the knot
        connected sum of the unknotted torus with a knotted sphere?
    \end{question}

    \begin{question}
        If the stable extendable subgroup $\esg^\stable_K$ equals $\mcg(T^2)$,
        does the singular genus $g^\star_K$ vanish for every slope?
    \end{question}

    \begin{question}
        If $K$ is incompressible, i.e.~$\partial X_K$ is $\pi_1$--injective in the
        complement $X_K$,
        is the stable extendable subgroup $\esg^\stable_K$ finite?
    \end{question}

    \begin{question}
        For plumbing knotted satellites, does the upper bound in Proposition~\ref{plumbingGenera} (2)
        realize the genus of the slope?
    \end{question}

\bibliographystyle{amsalpha}

\begin{thebibliography}{LNSW}

\bibitem[AHT]{AHT}
	I.~Agol, J.~Hass, W.~P.~Thurston,
	\textit{The computational complexity of knot genus and spanning area},
	Trans.~ Amer.~Math.~Soc.~\textbf{358} (2006), 3821--3850.

\bibitem[Ar]{Ar}
	E.~ Artin, 
	\textit{Zur Isotopie zweidimensionaler Fl{\"a}chen im $R_4$},
	Abh.~ Math.~ Sem.~ Univ.~Hamburg \textbf{4} (1925), 174--177.

\bibitem[BMS]{BMS}
	A.~M.~ Brunner, E.~ J.~ Mayland Jr., J.~Simon, 
	\textit{Knot groups in $S^4$ with nontrivial homology}, 
	Pacific J.~ Math.~ \textbf{103} (1982), 315--324.

\bibitem[Ca]{Ca}
    D.~Calegari,
    \textit{scl}, MSJ Memoirs, \textbf{20} (2009), Mathematical Society of Japan, Tokyo.

\bibitem[CKS]{CKS}
	S.~ Carter, S.~ Kamada, M.~Saito, 
	\textit{Surfaces in 4-Space},
	Encyclopaedia of Mathematical Sciences, \textbf{142}. 
	Low-Dimensional	Topology, III. Springer--Verlag, Berlin, 2004.

\bibitem[CS]{CS}
	J.~ S.~ Carter and M.~ Saito, 
	\textit{Knotted surfaces and their diagrams}, 
	Math.~ Surveys and Monogr.~ \textbf{55}, 
	Amer.~ Math.~ Soc., 1998.

\bibitem[DLWY]{DLWY}
    F.~Ding, Y.~Liu, S.-C.~Wang, J.-G.~Yao,
    \textit{A spin obstruction for codimension-two homeomorphism extension},
    Math.~Res.~Lett.~\textbf{19} (2012), 345--357.

\bibitem[Fo]{Fox}
	R.~ H.~ Fox, 
	\textit{A quick trip through knot theory},
	Topology of	3-Manifolds and Related Topics,
	(Proc.~The Univ.~of Georgia Institute, 1961), 
	pp.~ 120--167, 
	Prentice--Hall, Englewood Cliffs, NJ, 1962.

\bibitem[FQ]{FQ}
	M.~Freedman, F.~Quinn,
	\textit{Topology of 4-Manifolds}, 
	Princeton Mathematical Series 39, Princeton University Press, 1990.

\bibitem[FV]{FV}
	S.~Friedl, S.~Vidussi,
	\textit{The Thurston norm and twisted Alexander polynomials},
	preprint, available at \texttt{arXiv:1204.6456v2}.

\bibitem[Ga]{Ga}
    D.~Gabai,
    \textit{Foliations and the topology of $3$--manifolds},
    J.~Differential Geom.~\textbf{18} (1983), no.~3, 445--503.

\bibitem[GS]{GS}
    R.~E.~Gompf, A.~I.~Stipsicz,
    \textit{4-manifolds and Kirby calculus},
    Graduate Studies in Math.~20, Amer.~Math.~Soc., Providence, RI, 1999.

\bibitem[Go1]{Gordon-2Knot}
	C.~McA.~Gordon,
	\textit{Knots in the 4-sphere},
	Comment.~Math.~ Helv.~\textbf{51} (1976), 585--596.

\bibitem[Go2]{Gordon-secondHomology}
	\bysame, 
	\textit{Homology of groups of surfaces in the 4-sphere},
	Math.~ Proc.~ Cambridge Philos.~ Soc.~ \textbf{89}, (1981), 113--117.

\bibitem[Hil]{Hillman}
	J.~Hillman, 
	\textit{2-knots and their groups},
	Australian Mathemetical Society Lecture Series, Vol.~5,
	Cambridge Univ.~ Press, Cambridge, 1989.

\bibitem[Hir1]{Hi-T2}
    S.~Hirose,
    \textit{On diffeomorphisms over $T^2$--knots},
    Proc.~Amer.~Math.~Soc.~ \textbf{119} (1993), no.~3, 1009--1018.

\bibitem[Hir2]{Hi-trivial}
    \bysame,
    \textit{On diffeomorphisms over surfaces trivially embedded in the 4-sphere},
    Algebr.~Geom.~Topol. \textbf{2} (2002), 791--824.

\bibitem[HY]{HY}
    S.~Hirose,  A.~Yasuhara,
    \textit{Surfaces in 4-manifolds and their mapping class groups},
    Topol. \textbf{47} (2008), no.~1, 41--50.

\bibitem[HK]{HK}
	F.~Hosokawa, A.~Kawauchi,
	\textit{Proposals for unknotted surfaces in 4-spaces},
	Osaka J.~ Math.~\textbf{16} (1979), 233--248.

\bibitem[Kam1]{Kamada-essentialKnot}
    S.~Kamada,
    \textit{Orientable surfaces in the 4-space associated with nonorientable knotted surface},
    Math.~Proc.~ Camb.~Phil.~Soc.~\textbf{108} (1990), 299--306.

\bibitem[Kam2]{Kamada-groupCharacterization}
	\bysame, 
	\textit{A characterization of groups of closed orientable surfaces in 4-space}, 
	Topol.~ \textbf{33} (1994), 113--122.

\bibitem[Kam3]{Kamada-braid}
	\bysame,
	\textit{Braid and knot theory in dimension four}, 
	Math.~Surveys and Monogr.~ \textbf{95}, 
	Amer.~ Math.~ Soc., 2002.

\bibitem[Kan]{Kanenobu}
	T.~Kanenobu, 
	\textit{Groups of higher-dimensional satellite knots},
	J.~ Pure Appl.~ Algebra \textbf{28}(1983), 179--188.
	
\bibitem[KK]{KK}
	T.~Kanenobu, K.~Kazama,
	\textit{The peripheral subgroup and the second homology of the group of a knotted torus in $S^4$},
	Osaka J.~Math.~\textbf{31} (1994), 907--921.


\bibitem[Kaw]{Kawauchi}
	A.~Kawauchi, 
	\textit{A Survey of Knot Theory},
	Birkh{\"a}user Verlag, Berlin, 1996.


\bibitem[KSS]{KSS}
	A.~ Kawauchi, T.~ Shibuya, S.~ Suzuki,
	\textit{Descriptions on surfaces in four-space. I. Normal forms},
	Math.~ Sem.~ Notes Kobe Univ.~\textbf{10} (1982), 75--125.

\bibitem[Le]{Levine}
	J.~ Levine,
	\textit{Some results on higher dimensional knot groups},
	Lecture Notes in Mathematics \textbf{685}, pp.~243--269, 
	Springer--Verlag, Berlin, 1978.

\bibitem[Lit1]{Litherland-deformTwist}
	R.~A.~ Litherland, 
	\textit{Deforming twist-spun knots}, 
	Trans.~ Amer.~	Math.~ Soc.~ \textbf{250} (1979), 
	311--331.

\bibitem[Lit2]{Litherland-secondHomology}
	\bysame,
	\textit{The second homology of the group of a knotted surface},
	Quart.~J.~Math.~Oxford Ser.~(2) \textbf{32} (1981), 425--434.

\bibitem[Liv1]{Livingston-stablyIrreducible}
	C.~Livingston,
	\textit{Stably irreducible surfaces in $S^4$}
	Pac.~J.~Math.~\textbf{116} (1985), 77--84.

\bibitem[Liv2]{Livingston-indecomposable}
	\bysame,
	\textit{Indecomposable surfaces in 4-space},
	Pac.~J.~Math.~\textbf{132} (1988), 371--378.

\bibitem[Mae]{Maeda}
	T.~ Maeda, 
	\textit{On the groups with Wirtinger presentations}, 
	Math.~Sem.~ Notes Kobe Univ.~ \textbf{5} (1977), 347--358.

\bibitem[Ma]{Massey}
    W.~S.~Massey,
    \textit{Finite covering spaces of 2-manifolds with boundary},
    Duke Math.~J.~\textbf{41} (1974), 875--887.

\bibitem[MS]{MS}
    J.~Milnor, J.~Stasheff,
    \textit{Characteristic Classes},
    Annals of Mathematics Studies, no.~76, Princeton University Press; University of Tokyo Press, 1974.

\bibitem[Mo]{Mo}
    J.~M.~Montesinos,
    \textit{On twins in the four-sphere, I},
    Quart.~J.~Math.~Oxford Ser.~(2) \textbf{34} (1983), no.~134, 171--199.


\bibitem[OS]{OS}
	P.~Ozsv{\'a}th, Z.~Szab{\'o},
	\textit{Holomorphic disks and genus bounds},
	Geom.~Topol.~\textbf{8} (2004), 311--334.

\bibitem[Sc1]{Schubert-inequality}
    H.~Schubert,
    \textit{Knoten und Vollringe} (German),
    Acta Math.~\textbf{90} (1953), 131--286.
    
\bibitem[Sc2]{Schubert-genus}
	\bysame, 
	\textit{Bestimmung der Primfaktorzerlegung von Verkettungen},
	Math.~ Zeitschr.~\textbf{76} (1961), 116--148.
	
\bibitem[Se]{Seifert}
	H.~Seifert,
	\textit{{\"U}ber das Geschlecht von Knote} (German),
	Math.~Annalen \textbf{110} (1935), 571--592.

\bibitem[Sh]{Shinohara}
	Y.~Shinohara, 
	\textit{Higher dimensional knots in tubes},
	Trans.~ Amer.~Math.~ Soc.~ \textbf{161} (1971), 35--49.

\bibitem[Sk]{Skopenkov}
	A.~B.~Skopenkov,
	\textit{Knotting and embedding of manifolds in Euclidean spaces},
	Surveys in Contemporary Mathematics, London Math.~Soc.~Lecture Note
	\textbf{347}, pp.~248--342, Cambridge Univ. Press, Cambridge, 2008.	

\bibitem[Th]{Th}
	W.~P.~Thurston,
	\textit{A norm for the homology of 3-manifolds},
	Mem.~Amer.~Math.~Soc.~\textbf{339} (1986), 99--130.

\bibitem[Wu]{Wu}
	W.-T. Wu,
	\textit{On the isotopy of $C^{r}$-manifolds of dimension $n$ in euclidean $(2n+1)$-space},
	Sci.~Record (N.S.), \textbf{2} (1958), 271--275.
	
\bibitem[Ya]{Yajima-knotGroup}
	T.~ Yajima, 
	\textit{On the fundamental groups of knotted 2-manifolds in the 4-space},
	J.~ Math.~ Osaka City Univ.~ \textbf{13} (1962), 63--71.
%

\bibitem[Ze]{Ze}
	E.~C.~ Zeeman, 
	\textit{Twisting spun knots}, 
	Trans.~ Amer.~ Math.~ Soc.~\textbf{115} (1965), 471--495.



\end{thebibliography}

\end{document}